\documentclass[11pt]{article}
\usepackage[a4paper, margin=2.5cm]{geometry}
\usepackage[shortlabels]{enumitem}
\usepackage{amsmath,multicol}
\usepackage{amssymb,latexsym}
\usepackage{amsthm}
\usepackage{color}
\usepackage{cancel}
\usepackage{graphicx}
\usepackage[noadjust]{cite}
\usepackage{changes}
\usepackage{lineno}
\usepackage[hidelinks]{hyperref}
\usepackage{comment,cleveref}
\usepackage{amscd}
\usepackage{tikz}

\allowdisplaybreaks

\DeclareMathOperator{\GammaL}{\Gamma\mathrm{L}}

\newtheorem{theorem}{Theorem}[section]

\newtheorem{lemma}[theorem]{Lemma}
\newtheorem{corollary}[theorem]{Corollary}
\newtheorem{proposition}[theorem]{Proposition}
\newtheorem{result}[theorem]{Result}
\newtheorem{construction}[theorem]{Construction}

    \theoremstyle{definition}
\newtheorem{definition}[theorem]{Definition}
\newtheorem{example}[theorem]{Example}
\newtheorem{remark}[theorem]{Remark}
\newtheorem{notation}[theorem]{Notation}

\newcommand{\fqn}{\mathbb{F}_{q^n}}

\newcommand{\F}{{\mathbb F}}

\newcommand{\fq}{{\mathbb F}_{q}}

\renewcommand{\mod}{\hbox{{\rm mod}\,}}

  \renewcommand{\epsilon}{\varepsilon}
\newcommand{\PG}{\mathrm{PG}}
\newcommand{\pg}{\PG}

\newcommand{\e}{\mathbf e}
\DeclareMathOperator{\ww}{w}
\DeclareMathOperator{\GL}{GL}

\newcommand{\floor}[1]{\left \lfloor #1 \right \rfloor}

\title{On the minimum size of linear sets}
\author{
 Sam Adriaensen \\ \textit{Vrije Universiteit Brussel} \and Paolo Santonastaso \\ \textit{Università degli
Studi della Campania} \\ \textit{``Luigi Vanvitelli''}
}
\date{ }

\begin{document}
\maketitle

\begin{abstract}
 Recently, a lower bound was established on the size of linear sets in projective spaces, that intersect a hyperplane in a canonical subgeometry.
 There are several constructions showing that this bound is tight.
 In this paper, we generalize this bound to linear sets meeting some subspace $\pi$ in a canonical subgeometry.
 We obtain a tight lower bound on the size of any $\fq$-linear set spanning $\pg(d,q^n)$ in case that $n \leq q$ and $n$ is prime.
 We also give constructions of linear sets attaining equality in the former bound, both in the case that $\pi$ is a hyperplane, and in the case that $\pi$ is a lower dimensional subspace.
\end{abstract}


\section{Introduction}

Linear sets are certain point sets in projective spaces, generalizing the notion of a subgeometry.
They have proven themselves to be very useful in constructing interesting objects in projective spaces, such as blocking sets \cite{sziklai2008onsmall} and KM-arcs \cite{deboeck2016alinearsetview}, and have been used to construct Hamming and rank metric codes \cite{polverino2022divisible,Napolitano2023twopoint,alfarano2021linearcutting,sheekey2016new, sheekeyVdV,polverino2020connections,zini2021scattered}.
For a survey on linear sets, we refer the reader to \cite{lavrauw2015field,polverino2010linear}.

Given the usefulness of linear sets, their recent spurt in popularity within the field of finite geometry is far from surprising.
One of the most natural questions arising in the study of linear sets is establishing lower and upper bounds on their size.
There is a quite trivial upper bound on the size of linear sets, and the study of linear sets attaining equality in this bound can be traced back to a paper by Blokhuis and Lavrauw \cite{blokhuis2000scattered}.
However, finding good lower bounds on the size of linear sets seems to be a harder problem.
Yet it is an interesting endeavor, e.g.\ due to its connection with the weight distribution of linear rank metric codes \cite{PolvSanZullo}.

As a consequence of the celebrated result on the number of directions determined by a function over a finite field \cite{blokhuis1999number,ball2003number}, Bonoli and Polverino established a lower bound on the size of certain linear sets on a projective line.
More specifically, they proved the following result (for the definitions, we refer to \Cref{sec:Preliminaries}).

\begin{result}[{\cite[Lemma 2.2]{bonoli2005fqlinear}}]
 \label{res:MinSizeLineRankN}
 If $L_U$ is an $\fq$-linear set of rank $n$ on $\PG(1,q^n)$, and $L_U$ contains at least one point of weight 1, then $|L_U| \geq q^{n-1} + 1$.
\end{result}

De Beule and Van de Voorde managed to remove the condition on the rank from this bound.
We note that linear sets of rank greater than $n$ on $\PG(1,q^n)$ are not interesting to study, since they necessarily contain all the points of the projective line.
Hence it is natural to limit the study to linear sets whose rank is at most $n$.

\begin{result}[{\cite[Theorem 1.2]{debeule2019theminimumsize}}]
 \label{res:JanGeertruiLine}
 If $L_U$ is an $\fq$-linear set of rank $k$, with $1 < k \leq n$ on $\PG(1,q^n)$, and $L_U$ contains at least one point of weight 1, then $|L_U| \geq q^{k-1} + 1$.
\end{result}

Using an inductive argument, they obtained a bound on the size of a linear set in a higher dimensional projective space.
Using \Cref{lm:subgeometry}, which we prove later in this paper, this is equivalent to the following result.

\begin{result}[{\cite[Theorem 4.4]{debeule2019theminimumsize}}]
 \label{res:JanGeertrui}
 Let $L_U$ be an $\fq$-linear set of rank $k > d$ in $\PG(d,q^n)$.
 If $L_U$ meets some hyperplane $\Omega$ in a canonical $\fq$-subgeometry of $\Omega$, then
 \[
  |L_U| \geq q^{k-1} + q^{k-2} + \ldots + q^{k-d} + 1.
 \]
\end{result}

De Beule and Van de Voorde note directly after their statement of the above result that they would like to find lower bounds on linear sets satisfying less restrictive conditions.
Furthermore, Jena and Van de Voorde \cite[\S 2.5 (B)]{Jena2021onlinearsets} state that they believe the above lower bound to hold for all $\fq$-linear sets of rank $k$ that span $\pg(d,q^n)$, if $n$ is prime and $k \leq d+n$.

In this article we will generalize the above result by dropping the condition that $\Omega$ is a hyperplane.

\begin{theorem}
 \label{thm:OurBound}
 Let $L_U$ be an $\fq$-linear set of rank $k$ in $\PG(d,q^n)$.
 Suppose that there exists some $(r-1)$-space $\Omega$, with $r < k$, such that $L_U$ meets $\Omega$ in a canonical $\fq$-subgeometry of $\Omega$.
 Then
 \begin{equation*}
  |L_U| \geq q^{k-1} + \ldots + q^{k-r} + I_\Omega,
 \end{equation*}
 where $I_\Omega$ denotes the number of $r$-spaces through $\Omega$, containing a point of $L_U \setminus \Omega$.
\end{theorem}

The theorem leads us to wonder, given a linear set, how we can assure the existence of a large subspace intersecting it in a canonical subgeometry.
This question turns out to be closely related to studying which linear sets must certainly have a point of weight 1.
Csajbók, Marino, and Pepe \cite{csajbok2023maximum} recently proved the following seminal result.

\begin{result}[{\cite[Theorem 2]{csajbok2023maximum}}] \label{th:pointweightgreaterfield}
Let $L_U$ be an $\F_q$-linear set of $\PG(d,q^n)$ of rank $k \leq dn$, such that the following assumptions are satisfied:
\begin{enumerate}
    \item $n\leq q$;
    \item every point of $L_U$ has weight at least $w \geq 2$.
\end{enumerate}
Then there exists an integer $t$ with $w \leq t \mid n$ such that $L_U=L_{U'}$, with $U'=\langle U \rangle_{\F_{q^t}}$.
\end{result}

Especially when $n$ is prime, this is a powerful result.
In that case, a linear sets without points of weight 1 must coincide with a subspace as point sets.
This allows us to prove the following theorem.

\begin{theorem} \label{th:prime}
 Suppose that $n$ is a prime number with $n \geq q$.
 Let $L_U$ be an $\fq$-linear set in $\pg(d,q^n)$ spanning the whole space.
 Define $r = d - \floor{\frac{k-(d+2)}{n-1}}$.
 Then $L_U$ meets some $(r-1)$-space in a canonical subgeometry and
 \[
  |L_U| \geq q^{k-1} + \ldots + q^{k-r} + \frac{q^{n(d-r+1)}-1}{q^n-1}.
 \]
 Moreover, this lower bound is tight.
\end{theorem}

Note in particular that this confirms the previously mentioned belief of Jena and Van de Voorde \cite[\S 2.5 (B)]{Jena2021onlinearsets} -- that all $\fq$-linear sets of rank $k \leq d+n$ spanning $\pg(d,q^n)$, $n$ prime, satisfy the lower bound of \ref{res:JanGeertrui} -- in case that $n \leq q$.

Also in the case where $n$ is not prime, \Cref{th:pointweightgreaterfield} is interesting from the point of view of lower bounding the size of a linear set.
It ensures that we can take $r=1$ in \Cref{thm:OurBound}, i.e.\ $\Omega$ is a point, in case that $\fq$ is the maximum geometric field of linearity of $L_U$.

In $\PG(1,q^n)$, the bound of De Beule and Van de Voorde is tight.
For every rank $k \leq n$, there exist so-called $(k-1)$-clubs of rank $k$.
These linear sets contain (an abundance of) points of weight 1, and their size matches the bound in \Cref{res:JanGeertruiLine}.
Lunardon and Polverino \cite{lunardon2000blocking} provided the first less trivial family of linear sets of rank $n$ reaching equality in \Cref{res:MinSizeLineRankN}.
Their example was extended by Jena and Van de Voorde \cite{Jena2021onlinearsets} to a very large family of linear sets of general rank, attaining equality in \Cref{res:JanGeertruiLine}.
More recently, there have been other constructions of such linear sets, and partial classification results, see Napolitano et al.\ \cite{napolitano2022classification}.
Moreover, Jena and Van de Voorde generalized their constructions to higher dimensions, to obtain linear sets attaining equality in the bound of \Cref{res:JanGeertrui}, some of which also satisfy the conditions of \Cref{res:JanGeertrui} \cite[\S 2.5 (B)]{Jena2021onlinearsets}.

In this article, we study the construction by Jena and Van de Voorde in general dimension, and we provide a sufficient condition for these linear sets to satisfy the hypothesis of \Cref{res:JanGeertrui}.
We also generalize the construction of Napolitano et al.\ to higher dimensions.
Furthermore, we construct linear sets in $\PG(d,q^n)$ satisfying the conditions of \Cref{thm:OurBound}, and attaining equality in the corresponding bound, where $n$ is not prime.
The size of these linear sets is smaller than the bound from \Cref{res:JanGeertrui}, hence this illustrates the necessity of the conditions imposed in \Cref{res:JanGeertrui} in case $n$ is not prime.

\bigskip

\textbf{Structure of the paper.}
\Cref{sec:Preliminaries} contains preliminary results on linear sets.
\Cref{sec:GeneralBounds} contains the proof of \Cref{thm:OurBound,th:prime}.
In addition, we deduce from \Cref{thm:OurBound} that the rank of a linear set is determined by its size and the minimum weight of its points, and that it is spanned by its points of minimum weight.
In \Cref{sec:d-minSize} we discuss linear sets attaining equality in \Cref{res:JanGeertrui}.
More specifically, we show a sufficient condition for the minimum size linear sets of \cite{Jena2021onlinearsets} to satisfy the hypothesis of \Cref{res:JanGeertrui}, and we generalize the construction from \cite{napolitano2022classification} to higher dimension.
\Cref{sec:BelowBound} contains constructions of linear sets attaining equality in \Cref{thm:OurBound}.

\section{Preliminaries}
 \label{sec:Preliminaries}

Throughout this article, $q$ will always denote a prime power, and $\fq$ will denote the finite field of order $q$.
The $d$-dimensional projective space over $\fq$ will be denoted by $\PG(d,q)$.
If the projective space is constructed from a $(d+1)$-dimensional $\fq$-vector space $V$, and we want to emphasize the underlying vector space, we might also denote the projective space as $\PG(V,\fq)$.
We note that the number of points in $\PG(d,q)$ equals $\frac{q^{d+1}-1}{q-1} = q^d + q^{d-1} + \ldots + q + 1$.

\begin{notation}
 Throughout the article, when working in $\PG(d,q) = \PG(\fq^{d+1},\fq)$, we denote the vectors of $\fq^{d+1}$ as $(x_0, \dots, x_{d})$, i.e.\ we label the coordinate positions from 0 to $d$.
 The $i$\textsuperscript{th} standard basis vector will be denoted as
 \[
  \e_i = (0, \dots, 0, \underbrace 1 _{i^\text{th} \text{ position}}, 0 \dots, 0),
 \]
 and the corresponding point in $\PG(d,q)$ will be denoted as $E_i$.
\end{notation}

\subsection{Linear sets}

Let $V$ be a $(d+1)$-dimensional vector space over $\fqn$.
Then $V$ is also a $(d+1)n$-dimensional vector space over $\fq$.
Let $U$ denote an $\fq$-subspace of $V$.
Then
\[
 L_U=\{\langle u \rangle_{\fqn} \colon u\in U\setminus\{\mathbf{0}\}\}
\]
is a set of points in $\PG(d,q^n)$.
Sets of this type are called \emph{$\fq$-linear sets}, and the $\fq$-dimension of $U$ is called the \emph{rank} of $L_U$.

We note that if $U_1$ and $U_2$ are $\fq$-subspaces, and $L_{U_1}$ and $L_{U_2}$ are equal as point set in $\PG(d,q^n)$, this need not imply that $\dim_{\fq} U_1 = \dim_{\fq} U_2$.
Hence, the rank of a linear set $L_U$ is generally defined ambiguously by $L_U$ as point set in $\PG(d,q^n)$, without taking into account the underlying subspace $U$.

Given an $\fqn$-subspace $W \leq V$, we define the \emph{weight} of $\Omega = \PG(W,\fqn)$ to be
\[
 \ww_{L_U}(\Omega) = \dim_{\fq} (U \cap W).
\]
Note that $\ww_{L_U}(\Omega)$ equals the rank of the linear set $L_{U \cap W} = L_U \cap \Omega$.

For each $i \in \{1,\dots,n\}$, let $N_i(L_U)$ denote the number of points in $\PG(d,q^n)$ of weight $i$.
We will simply denote this as $N_i$ if $L_U$ is clear from context.
The numbers $N_1, \dots, N_n$ are called the \emph{weight distribution} of $L_U$.
In addition, the \emph{weight spectrum} of $L_U$ is the ordered tuple $(i_1,\dots,i_t)$ with $i_1 < \ldots < i_t$ and
\[
 \{ i_1, \dots, i_t \} = \{ \ww_{L_U}(P) : P \in L_U \} = \{ i \in \{1,\dots,n\} : N_i > 0\}.
\]
Let $k>0$ denote the rank of $L_U$.
Then the weight distribution satisfies the following properties.
\begin{align}
    |L_U| &=N_1+\ldots+N_n, \label{eq:pesicard} \\
    \sum_{i=1}^n N_i \frac{q^i-1}{q-1}&=\frac{q^k-1}{q-1}, \label{eq:pesivett} \\
    |L_U| &\leq \frac{q^k-1}{q-1}, \label{eq:card} \\
    |L_U| &\equiv 1 \ (\mod q). \label{eq:modqsize}
\end{align}

Let $T$ be an $\F_q$-subspace of $V$ with $\dim_{\F_q}(T)=r \leq d+1$.
If $\dim_{\F_{q^n}}(\langle T \rangle_{\F_{q^n}})=r$, we will say that $L_T \cong \PG(T,\F_q)=\PG(r-1,q)$ is an $\fq$-\emph{subgeometry} of $\PG(V,\F_{q^n})$ and $r$ is the \emph{rank} of the subgeometry $L_T$.
When $r=d+1$, we say that $L_T$ is a \emph{canonical subgeometry} of $\PG(V,\fqn)=\PG(d,q^n)$. 
Note that each point of a subgeometry $L_T$ has weight 1 and hence $\lvert L_T \rvert =\frac{q^{r}-1}{q-1}$, if $L_T$ has rank $r$.

Regarding the linearity of a linear set, we recall the following definitions explored in \cite{Jena2021thegeometricfield}.

\begin{definition}[{\cite[Definitions 1.1, 1.2]{Jena2021thegeometricfield}}]
An $\F_{q}$-linear set $L_U$ is  an \emph{$\F_{q^s}$-linear set} if $U$ is also an $\F_{q^s}$-vector space. 
We say that $\F_{q^s}$ is the \emph{maximum field of linearity} of $L_U$ if $s$ is the largest exponent such that $L_U$ is $\F_{q^s}$-linear.
\end{definition}

\begin{definition}[{\cite[Definitions 1.3, 1.4]{Jena2021thegeometricfield}}] \label{def:geometricfield}
An $\F_q$-linear set $L_U$ has \emph{geometric field of linearity} $\F_{q^s}$ if there exists an $\F_{q^s}$-linear set $L_{U'}$ such that $L_U=L_{U'}$. An $\F_q$-linear set $L_U$ has \emph{maximum geometric field of linearity} $\F_{q^s}$ if $s$ is the
largest integer such that $L_U$ has geometric field of linearity $\F_{q^s}$.
\end{definition}

The maximum field of linearity and the maximum geometric field of linearity do not always coincide.
Clearly if $L_U$ is an $\F_{q^s}$-linear set, it has geometric field of linearity $\F_{q^s}$, but the converse need not hold, see e.g.\ \cite[Example 1.5]{Jena2021thegeometricfield}.

\begin{remark} \label{rk:secantgeometricfield}
Note that if there exists a line $\ell$ that is $(q+1)$-secant to a linear set $L_U$, then by \eqref{eq:modqsize} $ L_U$ has maximum geometric field of linearity $\F_q$, see also \cite{Jena2021thegeometricfield}.
\end{remark}

We refer to \cite{polverino2010linear} and \cite{lavrauw2015field} for comprehensive references on linear sets.

\subsection{Subspaces of complementary weights}

Recently, there has been an interest in linear sets admitting subspaces of complementary weights (see below for the definition), due to their application in coding theory, see e.g. \cite{polverino2022divisible, Napolitano2023twopoint, Zullo2023multi-orbit}.
Linear sets on the projective line admitting two points of complementary weights have been studied in \cite{napolitano2022linearsets} (see also \cite{Jena2021thegeometricfield,napolitano2022classification}).
The higher dimensional analogue has been studied in \cite{Zullo2023multi-orbit}.
For the sake of completeness, we state the definition and prove the structural description of such linear sets here in full generality.

Call subspaces $W_1, \dots, W_m \leq_{q^n} \fqn^{d+1}$ \emph{independent} if each subspace $W_i$ intersects $\langle W_j : j \neq i \rangle_{\fqn}$ trivially, or equivalently if $\dim_{\fqn} \langle W_i : i=1, \dots, m \rangle = \dim_{\fqn} W_1 + \ldots + \dim_{\fqn} W_m$.

\begin{lemma}
 \label{lm:IndependentSubspaces}
 Let $W_1, \dots, W_m$ be independent subspaces in $\fqn^{d+1}$, and let $L_U$ be an $\fq$-linear set in $\PG(d,q^n)$ of rank $k$, that spans the entire space.
 Then
 \[
  \ww_{L_U}( \pg(W_1,\fqn)) + \ldots + \ww_{L_U}(\pg(W_m,\fqn)) \leq k.
 \]
 If equality holds, then $\fqn^{d+1} = W_1 \oplus \ldots \oplus W_m$.
\end{lemma}

\begin{proof}
Since no $W_i$ intersects the span of the others, it is permitted to consider the direct sum $W_1 \oplus \ldots \oplus W_m$.
Then
\begin{align*}
 k &= \dim_{\fq} U
 \geq \dim_{\fq} (U \cap (W_1 \oplus \ldots \oplus W_m))
 \geq \dim_{\fq} (U \cap W_1) + \ldots + \dim_{\fq} (U \cap W_m) \\
 &= \ww_{L_U}(\PG(W_1,\fqn)) + \ldots + \ww_{L_U}(\PG(W_m,\fqn))
\end{align*}
If equality holds, then
\[
 U \cap (W_1 \oplus \ldots \oplus W_m) = U.
\]
Since $W_1 \oplus \ldots \oplus W_m$ is an $\fqn$-subspace, and $\langle U \rangle_{\fqn} = \fqn^{d+1}$, we get that
\[
 W_1 \oplus \ldots \oplus W_m = \fqn^{d+1}. \qedhere
\]
\end{proof}

\begin{definition}
 If the subspaces $W_1, \dots, W_m$ attain equality in \Cref{lm:IndependentSubspaces}, we say that $\pg(W_1,\fqn),\dots,\pg(W_m,\fqn)$ are \emph{subspaces of complementary weight} (w.r.t.\ $L_U$).
\end{definition}

\begin{lemma}
 Let $L_U$ be an $\fq$-linear set spanning $\PG(d,q^n)$.
 Then there exist subspaces $\Omega_1, \dots, \Omega_m$ in $\PG(d,q^n)$  of complementary weight, with $\dim \Omega_i = d_i$ and $\ww_{L_U}(\Omega_i) = k_i$, if and only if $U$ is $\GL(d+1,q^n)$-equivalent to an $\fq$-subspace $U_1 \times \ldots \times U_m$, with each $U_i$ a $k_i$-dimensional $\fq$-subspace of $\fqn^{d_i+1}$ satisfying $\langle U_i \rangle_{\fqn} = \fqn^{d_i+1}$.
\end{lemma}

\begin{proof}
First suppose that such subspaces $\Omega_i = \PG(W_i,\fqn)$ exist.
Then there exists a map $\varphi \in \GL(d+1,q^n)$ such that $\varphi(W_1) = \langle \e_0, \dots, \e_{d_1} \rangle_{\fqn}$, $\varphi(W_2) = \langle \e_{d_1 + 1}, \dots, \e_{d_1 + d_2 + 1} \rangle_{\fqn}$, and so on.
As can be seen in the proof of the previous lemma,
\begin{align*}
 \varphi(U)
 = \varphi(U \cap W_1) \oplus \ldots \oplus \varphi(U \cap W_m),
\end{align*}
which equals $U_1 \times \ldots \times U_m$, with
\[
 U_i = \{ u \in \fqn^{d_i+1} : (0, \ldots, 0, u, 0, \dots, 0) \in \varphi(U) \cap \varphi(W_i) \}. 
\]
Clearly,
\[
 \dim_{\fq} U_i = \dim_{\fq} \varphi(U \cap W_i)
 = \dim_{\fq} U \cap W_i
 = \ww_{L_U}(\Omega_i) = k_i.
\]

Vice versa, suppose that $\varphi(U) = U_1 \times \ldots \times U_m$, with each $U_i$ a $k_i$-dimensional $\fq$-subspace of $\fqn^{d_i+1}$, for some $\varphi \in \GL(d+1,q^n)$.
Then define 
\begin{align*}
 W_1 = \langle \e_0, \ldots, \e_{d_1} \rangle_{\fqn}, &&
 W_2 = \langle \e_{d_1+1}, \ldots, \e_{d_1 + d_2 + 1} \rangle_{\fqn}, &&
 \dots
\end{align*}
Then clearly $\PG(W_1, \fqn), \dots, \PG(W_m,\fqn)$ are subspaces of complementary weights w.r.t.\ $L_U$.
Having subspaces of complementary weights is $\GL(d+1,q^n)$-invariant, which finishes the proof.
\end{proof}

\section{General bounds}
 \label{sec:GeneralBounds}
This section is devoted to proving \Cref{thm:OurBound,th:prime}. 
From \Cref{thm:OurBound} we derive that if a linear set $L_U$ contains a point of weight 1, its rank equals $\lceil \log_q(|L_U|) \rceil$, and $\langle L_U \rangle$ is spanned by the points of $L_U$ of weight 1.

\subsection{Proof of \texorpdfstring{\Cref{thm:OurBound}}{Theorem 1.4}}
 
De Beule and Van de Voorde proved the following bound.

\begin{result}[{\cite[Theorem 4.4]{debeule2019theminimumsize}}]
 Let $L_U$ be an $\fq$-linear set spanning $\PG(d,q^n)$ of rank $k$.
 Suppose that $L_U$ meets some hyperplane $\Omega$ in exactly $\frac{q^d-1}{q-1}$ points, spanning $\Omega$.
 Then
 \[
  |L_U| \geq q^{k-1} + q^{k-2} + \ldots + q^{k-d} + 1.
 \]
\end{result}

Note that if $d=1$, this result is exactly \Cref{res:JanGeertruiLine}.
We now prove that this result is equivalent to \Cref{res:JanGeertrui}.
This follows directly from the following lemma.

\begin{lemma}
 \label{lm:subgeometry}
Let $L_U$ be an $\fq$-linear set in $\PG(d-1,q^n)$, with $d \geq 2$.
Then $L_U$ spans $\PG(d-1,q^n)$ and satisfies $|L_U| = \frac{q^d-1}{q-1}$ if and only if $L_U$ is a canonical $\fq$-subgeometry.
\end{lemma}

\begin{proof}
If $L_U$ is a canonical subgeometry, then it immediately follows that $L_U$ spans the entire space, and $|L_U| = \frac{q^d-1}{q-1}$.
So suppose that $L_U$ spans the space, and that $|L_U| = \frac{q^d-1}{q-1}$.
We need to prove that all points of $L_U$ have weight 1.
Indeed in that case, by equations \eqref{eq:pesicard} and \eqref{eq:pesivett}, $L_U$ must then have rank $d$, which proves that $L_U$ is a canonical subgeometry.
So suppose by way of contradiction that $L_U$ has points of weight greater than 1.
Note that by Equations \eqref{eq:pesicard} and \eqref{eq:pesivett}, the rank of $L_U$ is some number $k > d$.
Let
\[
 \sigma = \langle P \in L_U \colon \ww_{L_U}(P) > 1 \rangle
\]
denote the subspace of $\PG(d-1,q^n)$ spanned by points of weight greater than 1.

Suppose that $\sigma$ is not $\PG(d-1,q^n)$. 
Then $L_U \not \subseteq \sigma$, and every point in $L_U \setminus \sigma$ is a point of weight 1.
Hence, there are at least $q^{k-1}$ points in $L_U \setminus \sigma$ corresponding to (necessarily distinct) points of weight 1 of $L_U$.
Thus, $|L_U| > q^{k-1} > \frac{q^d-1}{q-1}$ since $k > d$, a contradiction.

Hence $\sigma$ equals $\PG(d-1,q^n)$.
Let
\[
 m = \max_{P \in L_U} \ww_{L_U}(P)
\]
denote the maximum weight of the points of $L_U$.
Then we can choose $d$ independent points $P_1, \dots, P_d$ in $L_U$ such that $\ww_{L_U}(P_1) = m$, and $\ww_{L_U}(P_i) \geq 2$ for each $i$.
By \Cref{lm:IndependentSubspaces},
\[
 k \geq \sum_{i=1}^d \ww_{L_U}(P_i) \geq m + 2(d-1).
\]
Let $N_1,\ldots,N_m$ denote the weight distribution of $L_U$.
Then by Equations \eqref{eq:pesicard} and \eqref{eq:pesivett},
\[
 \frac{q^m-1}{q-1} |L_U| = \sum_{i=1}^m N_i \frac{q^m-1}{q-1}
 \geq \sum_{i=1}^m N_i \frac{q^i-1}{q-1} = \frac{q^k-1}{q-1}
 \geq \frac{q^{m+2(d-1)}-1}{q-1}.
\]
This implies that
\[
  \frac{q^d-1}{q-1} = |L_U| \geq \frac{q^{m+2(d-1)}-1}{q^m-1},
\]
which yields a contradiction if $d \geq 2$.
\end{proof}

We will now prove \Cref{thm:OurBound}.\\

\noindent \textit{Proof of \Cref{thm:OurBound}.}
Consider the $r$-spaces $\Pi_1,\Pi_2,\ldots$ of $\PG(d,q^n)$ through $\Omega = \PG(W,\fqn)$, with $\Pi_i=\PG(W_i,\F_{q^n})$, for each $i$.
We can order the $r$-spaces in such a way that $\Pi_i$ contains a point of $L_U \setminus \Omega$ if and only if $i \leq I_{\Omega}$.
Let 
\[
k_i= \dim_{\fq} (U \cap W_i)
\]
denote the rank of the $\fq$-linear set $L_{U \cap W_i}$.
Then the sets $W_i \cap U \setminus W$ partition the vectors in $U \setminus W$. 
Since $L_U$ intersects $\Omega$ in a canonical subgeometry, $\dim_{\fq} U \cap W = r$.
This yields 
\begin{align}
 \label{eq:sumki}
q^{k} - q^r &= \sum_{i = 1}^{I_\Omega} (q^{k_i} - q^r) &&
\implies && q^{k-r}  =1 + \sum_{i=1}^{I_\Omega} (q^{k_i - r} -1).
\end{align}

Analogously, the points of $\Pi_i \setminus \Omega$ partition the points of $L_U \setminus \Omega$. 
Note that for $i \leq I_{\Omega}$, we have that $L_U \cap \Pi_i= L_{U \cap W_i}$ is an $\F_q$-linear set in $\Pi_i$ of rank $k_i$, satisfying the hypothesis of \Cref{res:JanGeertrui}.
Hence,
\begin{align*}
 |L_U|
 & = |L_U \cap \Omega| + \sum_{i=1}^{I_\Omega} \left( |L_U \cap \Pi_i| - |L_U \cap \Omega| \right) \\
 & \geq \frac{q^r-1}{q-1} + \sum_{i=1}^{I_\Omega} \left( (q^{k_i-1} + \ldots + q^{k_i-r} + 1) - \frac{q^r-1}{q-1} \right) \\
 & = \frac{q^r-1}{q-1} + \sum_{i=1}^{I_\Omega} \left( q^{k_i-r} \frac{q^r-1}{q-1} - \frac{q^r-1}{q-1} + 1 \right) \\
 & = \frac{q^r-1}{q-1} \left(1 + \sum_{i=1}^{I_\Omega} (q^{k_i-r} - 1) \right) + I_\Omega.
\end{align*}
Using Equation \eqref{eq:sumki} this implies that
\[
\lvert L_U \rvert \geq \frac{q^r-1}{q-1} q^{k-r} + I_\Omega = q^{k-1}+q^{k-2}+\ldots+q^{k-r}+I_{\Omega}.
\tag*{\qedsymbol}
\]

\begin{remark}
If one wants to apply \Cref{thm:OurBound} to a particular linear set $L_U$, different choices of the $(r-1)$-space $\Omega$ can yield different bounds.
In other words, $I_\Omega$ need not be the same for all $(r-1)$-spaces meeting $L_U$ in a canonical $\fq$-subgeometry.
This is illustrated in the example below. 
\end{remark}

\begin{example}
Consider the $(n+1)$-dimensional $\fq$-subspace
\[
 U = \{ (x,x^q) : x \in \fqn\} \times \fq
\]
of $\fqn^3$.
Consider the corresponding $\fq$-linear set $L_U$ of rank $n+1$ in $\PG(2,q^n)$.
Every point of $L_U$ has weight 1, so we can apply \Cref{thm:OurBound} with $\Omega$ any point of $L_U$.
However, for a point $P \in L_U$, $I_P = q^{n-1} + 1$ if $P$ lies on the line $X_2=0$, and $I_P = \frac{q^n-1}{q-1}$ if $P$ does not lie on $X_2=0$.
These numbers are distinct if $n>2$.
\end{example}

We also remark that in \Cref{thm:OurBound} the number $I_\Omega$ of $r$-spaces through $\Omega$ containing a point of $L_U \setminus \Omega$ equals the size of a certain linear set.

\begin{definition}
 Consider an $\fq$-linear set $L_U$ in $\pg(V,\fqn)$ and a subspace $\Omega = \pg(W,\fqn)$.
 Let $\overline U$ denote the subspace $(U+W)/W$ of the quotient space $V/W$.
 Then the \emph{projection} of $L_U$ from $\Omega$ is the $\fq$-linear set $L_{\overline U}$ of $\pg(V/W,\fqn)$.
\end{definition}

\begin{lemma} \label{lem:projectionlinear}
 Suppose that $L_U$ is an $\fq$-linear set of rank $k$ in $\pg(V,\fqn)$ and let $L_{\overline U}$ be the projection of $L_U$ from an $(r-1)$-space $\Omega = \pg(W,\fqn)$.
 Then for each $\fqn$-subspace $W' \leq V$ through $W$,
 \[
  \ww_{L_{\overline U}}(\pg((W' + W)/W,\fqn)) = \ww_{L_U}(\pg(W',\fqn)) - \ww_{L_U}(\Omega).
 \]
 In particular, $L_{\overline U}$ has rank $k - \ww_{L_U}(\Omega)$, and $|L_{\overline U}|$ equals the number of $r$-spaces in $\pg(V,\fqn)$ through $\Omega$ that contain a point of $L_U \setminus \Omega$. Furthermore, if $L_U$ spans $\pg(V,\fqn)$, then $L_{\overline{U}}$ spans $\pg(V/W,\fqn)$.
\end{lemma}

\begin{proof}
 We can find $\fq$-subspaces $U_1, U_2, U_3$ of $U$ such that
 \begin{itemize}
  \item $U_1 = W \cap U$,
  \item $U_1 \oplus U_2 = W' \cap U$,
  \item $U_1 \oplus U_2 \oplus U_3 = U$.
 \end{itemize}
Then
 \begin{align*}
 \ww_{L_{\overline U}}(\pg((W'+W)/W,\fqn))
 & = \dim_{\F_q} (\overline U \cap ((W' + W)/W))
 = \dim_{\F_q} (\langle U,W \rangle_{\F_q} \cap W') - \dim_{\F_q} W \\
 & = \dim_{\F_q} (W \oplus U_2) - \dim_{\F_q}( W)
 = \dim_{\F_q} (U_2) \\
 & = \dim_{\F_q} (U_1 \oplus U_2) - \dim_{\F_q} (U_1)
 = \ww_{L_U}(\pg(W',\fqn)) - \ww_{L_U}(\Omega).
 \end{align*}
If we put $W' = V$, we see that $L_{\overline U}$ has rank $k- \ww_{L_U}(\Omega)$.
It also follows that the points of $L_{\overline U}$ are in 1-1 correspondence with the $(r+1)$-spaces $W'$ of $V$ with $\ww_{L_U}(\pg(W,\fqn)) > \ww_{L_U}(\Omega)$, which are exactly the $r$-spaces through $\Omega$ in $\pg(V,\fqn)$ containing a point of $L_U \setminus \Omega$.
\end{proof}

This tells us the following about the quantity $I_\Omega$ in \Cref{thm:OurBound}.

\begin{proposition} \label{prop:boundquotient}
In the hypothesis of Theorem \ref{thm:OurBound}, let $L_{\overline U}$ be the projection of $L_U$ from $\Omega$. Then
\begin{equation*}
 \lvert L_U \rvert \geq q^{k-1}+q^{k-2}+\ldots+q^{k-r}+\lvert  L_{\overline{U}} \rvert.
\end{equation*}
Moreover, $L_{\overline U}$ has rank $k-r$.
\end{proposition}



In the next sections, we will investigate linear sets attaining equality in the bound of \Cref{thm:OurBound}.
To this end, we introduce some relevant terminology.

\begin{definition}
Let $L_U$ be an $\fq$-linear set of rank $k$ in $\PG(d,q^n)$.
If
\[
 |L_U| = q^{k-1} + \ldots + q^{k-d} + 1,
\]
we say that $L_U$ is of \emph{$d$-minimum size}.
If there is some $(r-1)$-space $\Omega$ such that $L_U$ and $\Omega$ satisfy the hypothesis of \Cref{thm:OurBound}, and
\begin{equation*}
 |L_U| = q^{k-1} + \ldots + q^{k-r} + I_\Omega
 \leq q^{k-1} + \ldots + q^{k-d} + 1,
\end{equation*}
then we say that $L_U$ is of \emph{$(r,d,\Omega)$-minimum size}, or simply of \emph{$(r,d)$-minimum size}.
A linear set of $(d,d)$-minimum size, will also be called of \emph{proper $d$-minimum size}.
\end{definition}

\begin{remark} \label{rk:geometricfieldminimum}
By Remark \ref{rk:secantgeometricfield}, an $(r,d)$-minimum size linear set has maximum geometric field of linearity  $\F_q$ whenever $r \geq 2$.
\end{remark}

In the next proposition, we also prove that if a linear set is of $(r,d)$-minimum size, it is of $(r',d)$-minimum size for every $r' \leq r$.

\begin{proposition}
Let $L_U$ be an $(r,d)$-minimum size $\F_q$-linear set of rank $k$ in $\PG(d,q^n)$.
Then $L_U$ is of $(r',d)$-minimum size as well, for every $0 < r' \leq r$.    
\end{proposition}

\begin{proof}
 It is enough to prove the statement for $r'=r-1$.
 By hypothesis, we know that there is some $(r-1)$-space $\Omega=\PG(W,\F_{q^n})$ of $\PG(d,q^n)$ meeting $L_U$ in a canonical subgeometry, such that 
\begin{equation} \label{eq:inequality1}
 \lvert L_U \rvert=  q^{k-1}+q^{k-2}+\ldots+q^{k-r}+\lvert  L_{\overline{U}} \rvert,
\end{equation}
 where $ L_{\overline{U}} $ is the $\F_q$-linear set in $\PG(V/W,\F_{q^n})=\PG(d-r,q^n)$ defined by $\overline{U}=U+W \subseteq V/W$. 
 Let $ \Omega'= \PG(W',\F_{q^n}) $ be an $(r-2)$-space of $\Omega$ that meets $L_U$ in a canonical subgeometry.
 So, by Proposition \ref{prop:boundquotient}, we have that 
\begin{equation} \label{eq:inequality2}
 \lvert L_U \rvert \geq   q^{k-1}+q^{k-2}+\ldots+q^{k-r+1}+\lvert  L_{\overline{U'}} \rvert,
\end{equation}
 where $ L_{\overline{U'}} $ is the $\F_q$-linear set of rank $k-r+1$ in $\PG(V/W',\F_{q^n})=\PG(d-r+1,q^n)$ defined by $\overline{U'}=U+W' \subseteq V/W'$. Therefore, by \eqref{eq:inequality1}, it follows $q^{k-r}+\lvert  L_{\overline{U}} \rvert \geq \lvert  L_{\overline{U'}} \rvert$. 
 On the other hand, since $w_{L_U}(\Omega)=r$, we get $ w_{L_{\overline{U'}}}(\PG(W/W',\F_{q^n}))=1$ and so $L_{\overline{U'}}$ has a point of weight 1. 
 Now, by \Cref{prop:boundquotient}, we get that 
 \[
\lvert L_{\overline{U'}} \rvert \geq q^{k-r}+ \left\lvert L_{\overline{\overline{U'}}} \right\rvert
 \]
 with $\overline{{\overline {U'}}} = U/W' + W/W' \leq (V/W')/W$, which is equal to $\overline U = U + W \leq V/W$.
 Hence,
 \[
  |L_{\overline {U'}}| \geq q^{k-r}+ \lvert L_{\overline{U}} \rvert.
 \]
 Then, by \eqref{eq:inequality1}, equality holds in \eqref{eq:inequality2} and so $L_U$ is of $(r-1,d,\Omega')$-minimum size.
\end{proof}

We conclude this subsection by giving a sufficient condition to apply \Cref{res:JanGeertrui}.

\begin{theorem}
 \label{th:Rank=k+d-r}
Let $k,d$ and $r$ be non negative integers with $r<k,d$. Let $L_U$ be an $\F_q$-linear set in $\PG(d,q^n)$ of rank $k+d-r$ spanning $\PG(d,q^n)$.
Suppose that there is an $r$-space $\Omega$ of $\PG(d,q^n)$ such that $\ww_{L_U}(\Omega)=k$, and $\Omega$ contains an $(r-1)$-space $\Omega'$ that meets $L_U$ in a canonical $\fq$-subgeometry.
Then some hyperplane $\Pi$ of $\PG(d,q^n)$ meets $L_U$ in a canonical $\fq$-subgeometry, implying $\lvert L_U \rvert \geq q^{k+d-r-1}+q^{k+d-r-2}+\ldots+q^{k-r}+1$.
\end{theorem}

\begin{proof}
Suppose that $\PG(d,q^n) = \PG(V,\fqn)$, $\Omega = \PG(W,\fqn)$, and $\Omega' = \PG(W',\fqn)$.
Consider the projection of $L_U$ from $\Omega'$, which equals the linear set $L_{\overline U}$, with $\overline U = U + W' \subseteq V / W'$.
Write $P_0 = W/W'$, and choose any point $P_1 \in L_{\overline U} \setminus \{ P_0 \}$.
Since $L_U$ spans $\PG(d,q^n)$, we can extend $P_0,P_1$ to a subset $P_0, P_1, \dots, P_{d-r}$ of $L_{\overline U}$ that spans $\PG(V/W',\fqn)$.
Also, $\ww_{L_{\overline U}}(P_0) = k-r$, and the rank of $L_{\overline U}$ equals $k + d - 2r = (k-r) + (d-r)$.
Hence, by \Cref{lm:IndependentSubspaces},
\[
 (k-r) + (d-r) \geq \sum_{i=0}^{d-r} \ww_{L_{\overline U}}(P_i)
 = (k-r) + \sum_{i=1}^{d-r} \ww_{L_{\overline U}}(P_i),
\]
which implies that $\ww_{L_{\overline U}}(P_i) = 1$ for all $i \geq 1$, and $P_0, \ldots, P_{d-r}$ are points of complementary weights.
Therefore, $L_{\overline U}$ meets $\langle P_1, \ldots, P_{d-r} \rangle$ in a canonical subgeometry.
There is a unique $\fqn$-space $W''$ through $W'$ such that $\langle P_1, \ldots, P_{d-r} \rangle = \pg(W'' + W', \fqn)$.
It follows that $\pg(W'',\fqn)$ meets $L_U$ in a canonical subgeometry.
\end{proof}

\subsection{Proof of \texorpdfstring{\Cref{th:prime}}{Theorem 1.6}}

Our next goal it to prove \Cref{th:prime}.
Let us start by recalling two well-known results concerning linear sets.

\begin{lemma}
 \label{Lm:PointOfWeight1}
 Suppose that $L_U$ is an $\fq$-linear set of rank $k$ of $\pg(n,q)$.
 Let $m = \min \{ \ww_{L_U}(P) \, : \, P \in L_U \}$ denote the minimum weight of the points of $L_U$.
 Then there exists an $\fq$-linear set $L_{U'}$ of rank $k-m+1$ in $\pg(n,q)$ containing points of weight 1 such that $L_U$ and $L_{U'}$ coincide as point sets.
\end{lemma}

\begin{proof}
 Take a vector $u \in U$ such that $P = \langle u \rangle_{\fqn}$ has weight $m$ in $L_U$.
 Then there exists a $(k-m+1)$-dimensional $\fq$-subspace $U'$ of $U$ that intersects $\langle u \rangle_{\F_{q^n}}$ in a 1-dimensional subspace.
 Then $\ww_{L_{U'}}(P) = 1$.
 It remains to show that $L_U$ and $L_{U'}$ coincide as points sets.
 The inclusion $L_{U'} \subseteq L_U$ is evident.
 On the other hand, take a non-zero $v \in U$.
 Then, by Grassmann's identity
 \begin{align*}
  \ww_{L_{U'}}(\langle v \rangle_{\fqn})
  &= \dim_{\fq} ( \langle v \rangle_{\fqn} \cap U')
  = \dim_{\fq} ((\langle v \rangle_{\fqn} \cap U) \cap U') \\
  & = \dim_{\fq} (\langle v \rangle_{\fqn} \cap U) + \dim_{\fq} (U') - \dim_{\F_q} (\langle \langle v \rangle_{\F_{q^n}} \cap U, U' \rangle_{\F_q}) \\
  & \geq m + (k-m+1) - \dim_{\fq} (U) = 1.
 \end{align*}
This shows that $L_U \subseteq L_{U'}$.
Thus, $L_U$ and $L_{U'}$ coincide as point sets.
\end{proof}

\begin{lemma} \label{Lm:AllPoints}
 Let $L_U$ be an $\fq$-linear set in $\pg(d,q^n)$ of rank $k$.
 Then $L_U$ contains all points of $\pg(d,q^n)$ if and only if $k > dn$.
 In that case, $L_U$ contains a point of weight 1 if and only if $k = dn+1$.
\end{lemma}

\begin{proof}
 If $k \leq dn$, then $|L_U| \leq \frac{q^{dn}-1}{q-1}$ by (\ref{eq:card}), so it can't contain all points of $\pg(d,q^n)$.
 If $k > dn$, then $U$ intersects every $\fq$-subspace of $\fqn^{d+1}$ of dimension $n$ non-trivially by Grassmann's identity.
 In particular, $U$ intersects all one-dimensional $\fqn$-subspaces of $\fqn^{d+1}$ non-trivially.
 Therefore, $L_U$ contains all points of $\pg(d,q^n)$.
 
 If $k = dn+1$, then $L_U$ contains points of weight 1 by \Cref{Lm:PointOfWeight1}, since we have proven that $L_U$ can't coincide with a linear set of lower rank.
 If $k > dn+1$, then we can use Grassmann's identity again to prove that $L_U$ has no points of weight 1.
\end{proof}

Combining this with \Cref{th:pointweightgreaterfield}, we obtain the following lemma.

\begin{lemma}
 \label{lm:NPrimeNoPtsOfWt1}
 Suppose that $n$ is a prime number with $n \leq q$.
 Let $L_U$ be an $\fq$-linear set of rank $k$ spanning $\pg(d,q^n)$.
 Then $L_U$ does not contain any points of weight 1 if and only if $dn+2 \leq k \leq dn + n$.
\end{lemma}

\begin{proof}
 By \Cref{Lm:AllPoints}, it suffices to prove that $L_U$ contains a point of weight 1 if $k \leq dn$.
 Suppose the contrary, i.e.\ $k \leq dn$ and $L_U$ contains no points of weight 1.
 Then by \Cref{th:pointweightgreaterfield} and $n$ being prime, $L_U$ as point set must be a subspace of $\pg(d,q^n)$.
 Since $L_U$ spans the entire space, this means that $L_U$ contains all points of $\pg(d,q^n)$.
 Using \Cref{Lm:AllPoints}, we see that this contradicts $k \leq dn$.
\end{proof}

Before proving the lower bound of \Cref{th:prime}, let us provide examples attaining equality in the bound.

\begin{construction} \label{Constr:Prime}
 Choose integers $0 \leq r \leq d$.
 Let $U_1$ be a $k_1$-dimensional $\fq$-subspace of $\fqn^{d-r+1}$ with $(d-r)n + 2 \leq k_1 \leq (d-r+1)n$.
 Define $U = U_1 \times \fq^r \leq_{\fq} \fqn^{d+1}$.
 Then $L_U$ is an $\fq$-linear set spanning $\pg(d,q^n)$ of rank $k = k_1 + r$ of size
 \[
  q^{k-1} + \ldots + q^{k-r} + \frac{q^{(d-r+1)n}-1}{q^n-1}.
 \]
 Moreover, $r = d-\floor{\frac{k-(d+2)}{n-1}}$.
\end{construction}

\begin{proof}
 It follows immediately that $L_U$ is an $\fq$-linear set of rank $k$ spanning $\pg(d,q^n)$.
 Moreover, since $k = k_1 + r$, one can check that
 \[
  (d-r)n + 2 \leq k_1 \leq (d-r+1)n \implies d-r \leq \frac{k-(d+2)}{n-1} \leq d-r+1 - \frac{1}{n-1},
 \]
 which yields $r = d-\floor{\frac{k-(d+2)}{n-1}}$.

 It remains to determine the size of $L_U$.
 Let $u = (u_1,u_2) \in U_1 \times \fq^r$ be a non-zero vector of $U$.
 If $u_2 \neq \mathbf 0$, then clearly $\alpha (u_1, u_2) \in U$ if and only if $\alpha \in \fq$.
 Hence, the point $\langle (u_1,u_2) \rangle$ has weight 1 in $L_U$.
 To make such a vector, we have $q^{k_1}$ choices for $u_1$ and $q^r-1$ choices for $u_2$.
 Since this always gives points of weight 1, this gives us
 \[
  q^{k_1} \frac{q^r-1}{q-1} = q^{k-1} + \ldots + q^{k-r}
 \]
 points of $L_U$.

 The number of points of $L_U$ represented by a vector of the form $(u_1,\mathbf 0)$ equals the size of $L_{U_1}$, which equals the number of points of $\pg(d-r,q^n)$ by \Cref{Lm:AllPoints}.
 Putting this together, we find that the size of $L_U$ is as claimed in the lemma.
\end{proof}

Now we are ready to finish the proof of \Cref{th:prime}.

\begin{proof}[Proof of \Cref{th:prime}]
 Let $s$ denote the largest integer such that $L_U$ meets an $(s-1)$-space $\sigma$ in a canonical subgeometry. 
 Then the projection $L_{\overline U}$ of $L_U$ from $\sigma$ is an $\fq$-linear set of rank $k-s$ spanning $\pg(d-s,q^n)$.
 Moreover, \Cref{lem:projectionlinear} implies that $L_{\overline U}$ does not have any points of weight 1, since no $s$-space through $\sigma$ can intersect $L_U$ in a canonical subgeometry.
 By \Cref{lm:NPrimeNoPtsOfWt1}, this means that
 \[
  (d-s)n + 2 \leq k-s \leq (d-s)n + n,
 \]
 which is equivalent to
 \[
  s = d - \floor{\frac{k-(d+2)}{n-1}} = r.
 \]
 Hence, there exists an $(r-1)$-space $\rho$ intersecting $L_U$ in a canonical subgeometry.
 Moreover, the projection $L_{\overline U}$ of $L_U$ from $\rho$ is an $\fqn$-linear set in $\pg(d-r,q^n)$ whose rank exceeds $(d-r)n$.
 By \Cref{Lm:AllPoints}, $L_{\overline U}$ contains all points of $\pg(d-r,q^n)$ and
 \[
  |L_U| \geq q^{k-1} + \ldots + q^{k-r} + \frac{q^{n(d-s+1)}-1}{q^n-1}
 \]
 by \Cref{thm:OurBound} and \Cref{prop:boundquotient}.
 The bound is tight by \Cref{Constr:Prime}.
\end{proof}

\subsection{Consequences of Theorem \texorpdfstring{\ref{thm:OurBound}}{1.4} }

When $r=1$, \Cref{thm:OurBound} looks as follows.

\begin{corollary} \label{cor:pointweightonesize}
Let $L_U$ be an $\F_q$-linear set of rank $k \geq 2$ in $\PG(d,q^n)$, admitting at least one point of weight 1. Let $I$ be the number of secant lines through some point of weight 1. Then $\lvert L_U \rvert \geq q^{k-1}+I$.
\end{corollary}

In particular, this result implies that the rank of a linear set is determined by its size and the minimum weight of its points.

\begin{proposition}
 Let $L_U$ be an $\fq$-linear set spanning $\PG(d,q^n)$, containing more than one point.
 Denote $m = \min_{P \in L_U} \ww_{L_U}(P)$.
 Then the rank of $L_U$ is the unique integer $k$ satisfying
 \[
  q^{k-m} + 1 \leq |L_U| \leq \frac{q^k-1}{q^m-1},
 \]
 i.e.\ $k = \lceil \log_q(|L_U|) \rceil + m - 1 = \lfloor \log_q(|L_U|) \rfloor + m$.
\end{proposition}

\begin{proof}
 By \Cref{Lm:PointOfWeight1}, $L_U$ coincides as point set with an $\fq$-linear set $L_{U'}$ of rank $k-m+1$, containing points of weight 1.
 By \Cref{thm:OurBound}, $|L_{U'}| \geq q^{k-m}+1$.
 The lower bound follows from \Cref{Lm:PointOfWeight1} and \Cref{cor:pointweightonesize}.
 By \Cref{eq:pesivett},
 \[
  (q^m-1) \lvert L_U \rvert=(q^m-1) \sum_{i=m}^n N_i \leq \sum_{i=m}^n N_i(q^i-1) = q^k-1 . \qedhere
 \]
\end{proof}

Another consequence of \Cref{cor:pointweightonesize} is that any $\fq$-linear set is spanned by its points of minimum weight, (cf.\ \cite[Lemma 2.2]{bonoli2005fqlinear} for linear sets on $\pg(1,q^n)$).

\begin{proposition}
If an $\fq$-linear set $L_U$ spans $\pg(d,q^n)$, then its points of minimum weight also span $\pg(d,q^n)$.
\end{proposition}

\begin{proof}
Suppose that $L_U$ is an $\fq$-linear set of rank $k$, spanning $\pg(d,q^n)$, and denote $m = \min_{P \in L_U}{\ww_{L_U}(P)}$.
By Corollary \ref{cor:pointweightonesize}, $|L_U| > q^{k-m}$.
Now assume that the points of weight $m$ of $L_U$ lie in a hyperplane $\pi=\PG(W,\F_{q^n})$ of $\PG(d,q^n)$.
Suppose that $U_1 = U \cap W$.
Then there exits a subspace $U_2$ of $U$ such that $U = U_1 \oplus U_2$.
Now let $U_1'$ be an $\fq$-subspace of $U_1$ of codimension $m-1$, and let $U_2'$ be an $\fq$-subspace of $U_2$ of codimension 1.
Let $U' = U_1' \oplus U_2'$, then $L_{U'}$ and $L_U$ coincide as point sets.
Indeed, if $P \in L_U \cap \pi$, then $\ww_{L_U}(P) = \ww_{L_{U_1}}(P) \geq m$.
As in the proof of \Cref{Lm:PointOfWeight1}, this implies that $\ww_{L_{U'}}(P) = \ww_{L_{U_1'}}(P) \geq 1$.
If $P \in L_U \setminus \pi$, then $\ww_{L_U}(P) \geq m+1$, and as in the proof of \Cref{Lm:PointOfWeight1}, $\ww_{L_{U'}}(P) \geq 1$.
But
\[
 \dim_q U' = (\dim_q U_1 - (m-1)) + (\dim_q U_2 - 1) = k-m.
\]
Hence, by \Cref{eq:card}, $|L_U| = |L_{U'}| \leq \frac{q^{k-m}-1}{q-1} < q^{k-m}$, a contradiction.
\end{proof}

\section{Constructions of \texorpdfstring{$d$}{d}-minimum size linear sets}
 \label{sec:d-minSize}

\subsection{Exploring the Jena-Van de Voorde construction}

Recently, Jena and Van de Voorde constructed $d$-minimum size linear sets admitting points of complementary weights, and they completely determined their weight spectrum and weight distribution.
Recall that if $\lambda \in \fqn$, then the \emph{degree} of $\lambda$ over $\fq$ equals the degree of the minimal polynomial of $\lambda$ over $\fq$, or equivalently the smallest integer $t$ such that $\lambda \in \F_{q^t}$.

\begin{construction}[{\cite[Theorem 2.17]{Jena2021onlinearsets}}] 
 \label{cs:constructionVdV}
Suppose that $\lambda \in \fqn$ has degree $t > 1$ over $\fq$.
Choose positive integers $k_0 \geq \ldots \geq k_d$ such that $k_0 + k_1 \leq t+1$.
Define
\begin{align*}
 JV_{q,n}(\lambda,t;k_0,\dots,k_d) &= \langle 1 ,\lambda,\ldots, \lambda^{k_0-1} \rangle_{\F_q} \times \ldots \times \langle 1 ,\lambda,\ldots, \lambda^{k_d-1} \rangle_{\F_q} \\
   &= \{ (f_0(\lambda), \dots, f_d(\lambda)) \colon f_i \in \fq[X], \, \deg(f_i) < k_i \}.
\end{align*}
Then $L_{JV_{q,n}(\lambda,t;k_0,\dots,k_d)}$ is a $d$-minimum size $\fq$-linear set in $\pg(d,q^n)$ of rank {$k_0 + \ldots + k_{d}$}.
\end{construction}

Note that since $JV_{q,n}(\lambda,t;k_0,\ldots,k_d)$ is a Cartesian product of $\F_q$-subspaces of $\F_{q^n}$, it indeed admits points of complementary weights.

Before proceeding, we make some conventions regarding polynomials.

\begin{definition}
 Given two polynomials $f, g \in \fq[X]$, let $\gcd(f,g)$ denote the unique monic polynomial of maximal degree that divides $f$ and $g$.
 We call $f$ and $g$ \emph{coprime} if $\gcd(f,g) = 1$.
 Furthermore, we will use the convention that the degree of the zero polynomial is $-\infty$, so that the equality $\deg(f \cdot g) = \deg f + \deg g$ still holds if $f$ or $g$ is the zero polynomial.
\end{definition}

\begin{remark}[{\cite[Remark 2.19]{Jena2021onlinearsets}}]
 \label{rk:WeightSpectrumJena}
Jena and Van de Voorde also determined the weight spectrum of the above linear set.
It is $(1,\dots,k_0)$ if $k_1 = k_0$, and $(1,\dots,k_1,k_0)$ if $k_1 < k_0$, in which case $E_0$ is the unique point of weight $k_0$.
They also described the weight distribution, but since it is rather involved, we omit it here.
It follows from their arguments that if $\gcd(f_0, \dots, f_d) = 1$, then
\begin{equation} \label{eq:weightjenavandevoorde}
 \ww_{L_U}\left( \langle f_0(\lambda), \dots, f_d(\lambda) \rangle_{\fqn} \right)
 = \min_{0 \leq i \leq d} \{k_i - \deg(f_i)\}.
\end{equation}
This makes it relatively easy to determine $N_i$ for some large values of $i$.
For instance, let \[
 U = JV_{q,n}(\lambda,t;k_0,\dots,k_d) \subseteq \F_{q^n}^{d+1},
\]
and assume that $k_1 < k_0$. As stated above, $E_0$ is the unique point of weight $k_0$, and the second largest weight of $L_{U}$ is $k_1$.
We can determine $N_{k_1}(L_U)$. 
Let $m$ denote the number of indices $j$ with $k_j = k_1$, i.e.\ $k_1= \ldots = k_{m} > k_{m+1}$. Let $P=\langle f_0(\lambda), \dots, f_d(\lambda) \rangle_{\fqn} \in L_U$, with $\gcd(f_0, \dots, f_d) = 1$. Then, by \eqref{eq:weightjenavandevoorde}, $P$ has weight $k_1$ if and only if $\deg(f_0) \leq k_0-k_1$, $\deg(f_i) \leq 0$, for $i=1,\ldots,m$ and $f_i = 0$, for $i>m$ and there exists some $j \in \{1,\ldots,m\}$ such that $\deg(f_j) > 0$.
Then, 
\[
 N_{k_1}(L_{U}) = q^{k_0 - k_1 + 1}\frac{q^{m}-1}{q-1}.
\]
\end{remark}

The above construction has the following consequence on the existence of $d$-minimum size linear sets in $\pg(d,q^n)$. 

\begin{corollary} [{\cite[Corollary 2.18]{Jena2021onlinearsets}}] 
 \label{cor:maximumsizeminimum}
There exists a $d$-minimum size $\F_q$-linear set of rank $k$ in $\PG(d,q^n)$ whenever
\[
 d < k \leq \begin{cases}
  (d+1)\frac{n+1}{2} & \text{if $n$ is odd,} \\
  (d+1)\frac{n}{2} + 1 & \text{if $n$ is even.}
 \end{cases} 
\]
\end{corollary}


We now present a sufficient condition for the linear set of \Cref{cs:constructionVdV} to be of proper $d$-minimum size.

\begin{theorem} \label{prop:properminimumjena}
Consider $U = JV_{q,n}(\lambda,t;k_0,\dots,k_d)$ as in \Cref{cs:constructionVdV}.
Suppose that there exist pairwise coprime polynomials $g_0, \dots, g_d \in \fq[X]$ such that for each $i$, $\deg g_i = k_i - 1$.
If $k_0 + \ldots + k_d \leq t+d$, then $L_U$ is of proper $d$-minimum size.
\end{theorem}

\begin{proof}
By \Cref{cs:constructionVdV}, we know that $L_U$ is an $\fq$-linear set in $\PG(d,q^n)$ of rank $k=k_0+\ldots+k_{d}$ of $d$-minimum size.
So it remains to prove that there exists a hyperplane of $\PG(d,q^n)$ meeting $L_U$ in a canonical subgeometry.
Consider the points $P_i=\langle \e_0 + g_i(\lambda) \e_i \rangle_{\F_{q^n}}$ for $i=1,\dots,d$.
Clearly, $P_1, \dots, P_d$ are independent, hence they span a hyperplane.
Define the polynomial
\[
 G(X) = \prod_{i=1}^d g_i(X).
\]
Note that for each $i \geq 1$, the polynomial
\[
 \left( \frac G {g_i} \right)(X) = \prod_{\substack{j=1 \\ j \neq i}}^d g_j(X)
\]
is well-defined.
Then the equation of the hyperplane $\Pi = \langle P_1, \dots, P_d \rangle_{\fqn}$ of $\pg(d,q^n)$ is
\begin{equation}
 G(\lambda) X_0 = \sum_{i=1}^d \left( \frac G {g_i} \right) (\lambda) X_i.
 \label{eq:Hyperplane}
\end{equation}
Let $k = k_0 + \ldots + k_d$ denote the rank of $L_U$.
Then
\begin{align*}
 \deg G = \sum_{i=1}^d (k_i-1) = k - k_0 - d < t,
\end{align*}
hence $G(\lambda) \neq 0$, and Equation \eqref{eq:Hyperplane} does indeed define a hyperplane.
Now take a non-zero vector $v = (f_0(\lambda), \dots, f_d(\lambda)) \in U$, and suppose that $\langle v \rangle_{\fqn} \in \Pi$.
Then
\begin{equation}
 G(\lambda) f_0(\lambda) = \sum_{i=1}^d \left( \frac G {g_i} \right) (\lambda) f_i(\lambda).
 \label{eq:PointInHyperplane}
\end{equation}
Every term in \Cref{eq:PointInHyperplane} is a polynomial in $\lambda$, and
\begin{align*}
 \deg (G f_0) & = \deg G + \deg f_0 = (k-k_0-d) + \deg f_0 < t, \\
 \deg ((G/g_i) f_i) &= \deg G + \deg f_i - \deg(g_i) \leq \deg(G) < t.
\end{align*}
Since $1, \lambda, \dots, \lambda^{t-1}$ are $\fq$-linearly independent, \Cref{eq:PointInHyperplane} implies that
\begin{equation*}
 G(X) f_0(X) = \sum_{i=1}^d \left( \frac G {g_i} \right)(X) f_i (X).
\end{equation*}

On the one hand, this implies that $f_0$ is a constant polynomial.
Otherwise, the left-hand side has degree greater than $\deg(G)$, but the degree of the right-hand side is at most $\deg(G)$, a contradiction.
On the other hand, for each $i$,
\[
 g_i(X) \mid \left( G(X) f_0(X) - \sum_{1 \leq j \neq i} \left(\frac G {g_j} \right) (X) f_j(X) \right) = \left(\frac G {g_i}\right)(X) f_i(X).
\]
Since $G/g_i$ is coprime with $g_i$, and $\deg(f_i) \leq \deg (g_i)$ this is only possible if $f_i$ is an {$\F_q^*$}-multiple of $g_i$.
Hence,
\[
 v = (\alpha_0, \alpha_1 g_1(\lambda), \dots, \alpha_d g_d(\lambda)),
\]
for some scalars $\alpha_0, \dots, \alpha_d \in \fq$.
Moreover, since $\langle v \rangle_{\fqn} \in \Pi$, $\alpha_0 = \alpha_1 + \ldots + \alpha_d$.
Hence, $L_U$ intersects $\Pi$ in the linear set $L_W$, with
\[
 W = \left\{ \sum_{i=1}^d \alpha_i (\e_0 + g_i(\lambda) \e_i) \colon \alpha_i \in \fq \right\}.
\]
Therefore, $L_U$ intersects $\Pi$ in a canonical subgeometry.
\end{proof}

A sufficient condition to ensure the existence of pairwise coprime polynomials $g_0,\ldots,g_{d} \in \F_q[X]$ such that $\deg(g_i)=k_i-1$, is to choose the size of the ground field large enough. 

\begin{proposition}
    Consider $U = JV_{q,n}(\lambda,t;k_0,\dots,k_d)$ as in \Cref{cs:constructionVdV}, with $k_0 + \ldots + k_d \leq t+d$.
Assume that 
\[
\sum_{i=0}^d k_i-d-1 \leq q.
\]
Then $L_U$ is of proper $d$-minimum size.
\end{proposition}
\begin{proof}
    By the hypothesis, we can consider $d+1$ subsets $S_0,\ldots,S_{d}$ of $\F_q$ that are pairwise disjoint and such that $\lvert S_i \rvert=k_i-1$, for each $i \in \{0,\ldots,d\}$. Then, we can define $g_i(x)=\prod_{\alpha \in S_i}(x-\alpha_i)$. So the assertion follows by \Cref{prop:properminimumjena}.
\end{proof}

Another sufficient condition to ensure the existence of pairwise coprime polynomials $g_0,\ldots,g_{d} \in \F_q[X]$ such that $\deg(g_i)=k_i-1$, is that the $g_i$'s are different monic irreducible polynomials over $\F_q$.
It is well known, see e.g.\ \cite[Theorem 3.25]{lidl1997finite}, that the number of monic irreducible polynomials of degree $s$ over the finite field $\F_q$ is given by Gauss’s formula 
\[
\frac{1}{s} \sum_{h \mid s} \mu (s/h)q^h,
\]
where $h$ runs over the set of all positive divisors of $s$ and $\mu$ denotes the Möbius function.

\begin{remark}
We note the following lower bound on the number of monic irreducible polynomials of degree $s$ over $\F_q$, see e.g.\ \cite{blaser2012computational}:
\[
\frac{1}{s} \sum_{h \mid s} \mu \left(s/h\right)q^h \geq \frac{q^s-2q^{s/2}}{s}.
\]
\end{remark}

So we get the following corollary.

\begin{corollary} \label{cor:necconditionhyper}
Consider $U = JV_{q,n}(\lambda,t;k_0,\dots,k_d)$ as in \Cref{cs:constructionVdV}, with $k_0 + \ldots + k_d \leq t+d$.
For each $s = 1, \dots, t$, suppose that
\[
 |\{ i \colon k_i-1 = s \} | \leq \frac{q^s-2q^{s/2}}{s}.
\]
Then $L_U$ is of proper $d$-minimum size.
\end{corollary}

Clearly, if the rank of a linear set $L_U$ obtained from  \Cref{cs:constructionVdV} is greater than $n+d$, then every hyperplane has weight at least $d+1$ in $L_U$, so $L_U$ cannot be of proper $d$-minimum size.
In case the rank exceeds $n+d$, we can prove that $L_U$ is of $(1,d)$-minimum size under some constraints on the rank.

\begin{proposition}
Let $U = JV_{q,n}(\lambda,t;k_0, \dots, k_d)$ be as in \Cref{cs:constructionVdV}.
If $k_0 + k_{d-1} + k_d \leq t + 2$, then $L_U$ is a $(1,d)$-minimum size $\fq$-linear set.
\end{proposition}

\begin{proof}
Let 
\[
 U'=\{(f_0(\lambda), \ldots, f_{d-1}(\lambda)+\lambda^{k_{d-1}-1}f_{d}(\lambda), f_{d}(\lambda)) \colon f_i \in \fq[X], \, \deg(f_i) < k_i \}.
\]
Then $U'$ is $\mathrm{GL}(d+1,q^n)$-equivalent to $U$ via the $\fqn$-linear map
\[
 \varphi: v = (v_0, \dots, v_d) \mapsto v + v_d \lambda^{k_{d-1}-1} \e_{d-1}.
\]
The point $\langle (0,\ldots,0,-\lambda^{k_{d-1}-1},1) \rangle_{\F_{q^n}}$ has weight $1$ in $L_U$ and it is mapped to point $E_{d}$ by $\varphi$.
So $E_{d}$ has weight $1$ in $L_{U'}$.
We prove that 
$\lvert L_{U'} \rvert=q^{k-1}+\lvert L_{\overline{U}} \rvert$,
where $\overline{U}=U'+E_d \leq_q \F_{q^n}^{d+1}/E_d$.
Note that $\F_{q^n}^{d+1}/E_d$ can be identified with $\F_{q^n}^{d}$ and 
$\overline{U}=U'+E_d \leq_q \F_{q^n}^{d+1}/E_d$ with 
\[
 \overline{U} = JV_{q,n}(\lambda,t;k_0, \dots, k_{d-2}, k_{d-1} + k_d - 1).
\]
By hypothesis $k_{0}+k_{d-1}+k_{d}-1 \leq t+1$, and so $k_0, \dots, k_{d-2}, k_{d-1} + k_d - 1$ indeed satisfy the hypothesis of \Cref{cs:constructionVdV} when rearranged in descending order.
Therefore, $\lvert L_{\overline{U}} \rvert=q^{k-2}+\ldots+q^{k-d}+1$.
Since $\lvert L_{U} \rvert=\lvert L_{U'} \rvert=q^{k-1}+\ldots+q^{k-d}+1=q^{k-1}+ |L_{\overline{U}}|$, we have the assertion.
\end{proof}

The above proposition together with Corollary \ref{cor:maximumsizeminimum}, allows to construct $(1,d)$-minimum size linear sets whose ranks exceed $n+d$. 
\begin{corollary}
There exist $(1,d)$-minimum size $\F_q$-linear sets in $\PG(d,q^n)$, $d \geq 2$, of rank $k$, whenever
\[
 d < k \leq \begin{cases}
  d \frac{n+1}2 +1 & \text{if $n$ is odd}, \\
  d \frac n 2 + 2 & \text{if $n$ is even.}
 \end{cases}
\]
\end{corollary}

\subsection{Generalizing the Caserta construction}

In \cite[Theorem 4.1]{napolitano2022classification}, a construction is given of linear sets on the projective line, based on the more general framework exploited in \cite{Gacs2003onarcs} and \cite{polverino1999blocking}.
In this subsection, we generalize this to higher dimensions.
The construction starts from an $\fq$-linear set $L_{U'}$ in $\PG(d,q^t)$, and yields an $\fq$-linear set in $\PG(d,q^{st})$.
Moreover, the weight distribution of $L_U$ is completely determined by the weight distribution of $L_{U'}$.

\begin{construction}
 \label{th:contructionfromqt}
Suppose that $n =s t$ with $s,t>1$.
Let $U'$ be an $\fq$-subspace of $\F_{q^t}^{d+1} \subseteq \F_{q^n}^{d+1}$ with $\dim_{\fq}(U')=k' > 0$.
Let $Z$ be an $\F_{q^t}$-subspace of $\F_{q^n}$ of dimension $r>0$, such that $1 \notin Z$.
Define 
\[ 
 C_{q,s,t}(Z,U') :=\{(z+ u_0, u_1,\ldots,u_{d}) \colon z \in Z, (u_0,\ldots, u_{d}) \in U'\} \subseteq \F_{q^n}^{d+1},
\]
which we will simply denote by $U$.
Then
\begin{enumerate}[(1)]
 \item the $\F_q$-linear set $L_U \subseteq \PG(d,q^n)$ has rank $r t + k'$,
 \item $|L_U| = q^{r t} |L_{U'} \setminus \{ E_0 \} | + 1$,
 \item $\ww_{L_U}(E_0) = r t + \ww_{L_{U'}}(E_0)$,
 \item $N_i(L_U) = q^{r t} ( N_i(L_{U'}) - \delta_{i,\ww_{L_{U'}}(E_0)} ) +  \delta_{i,\ww_{L_{U}}(E_0)}$,
\end{enumerate}
where $\delta_{i,j}$ denotes the Kronecker symbol.
\end{construction}

\begin{proof}
(1) Since $Z$ is an $\F_{q^t}$-subspace of $\fqn$, and $1 \notin Z$, $Z \cap \F_{q^t} = \{ 0 \}$.
Furthermore, since $U'$ is an $\F_q$-subspace of $\F_{q^t}^{d+1}$, $Z \cap \{u_0 : (u_0, \dots, u_d) \in U' \} = \{ 0 \}$.
Hence,
\[
 U = (Z \times \{ 0 \}^d) \oplus_{\fq} U'.
\]
Therefore,
\[
 \dim_{\fq} U = \dim_{\fq} Z + \dim_{\fq} U' = rt + k'.
\]

(3) Similarly, 
\begin{align*}
 \ww_{L_{U}}(E_0)
 & = \dim_{\fq} \left( Z \oplus_{\fq} \{ u_0 : u_0 \e_0 \in U' \} \right)
 = \dim_{\fq} Z + \dim_{\fq}(\{ u_0 : u_0 \e_0 \in U' \}) \\
 & = r t + \ww_{L_{U'}}(E_0).
\end{align*}

(2,4)
Suppose that
\[
 \langle z \e_0 + u \rangle_{\fqn} = \langle z' \e_0 + v \rangle_{\fqn},
\]
with $z, z' \in Z$, and $u, v \in U' \setminus \langle \e_0 \rangle_{\fqn}$.
Then $z \e_0 + u = \alpha (z' \e_0 + v)$ for some $\alpha \in \fqn$.
Since $u, v$ are not multiples of $\e_0$, there must exist some position $j > 0$ such that $u_j, v_j \neq 0$.
This implies that $\alpha = v_j / u_j \in \F_{q^t}$.
We also have that $z + u_0 = \alpha(z' + v_0)$, hence
\[
 z - \alpha z' = \alpha v_0 - u_0
\]
Recall that $Z$ is an $\F_{q^t}$-subspace, and that $u_0, v_0, \alpha \in \F_{q^t}$.
Therefore, the left-hand side of the above equality is in $Z$, and the right-hand side is in $\F_{q^t}$.
Since $Z \cap \F_{q^t} = \{0\}$, this implies that $z = \alpha z'$ and therefore $u = \alpha v$.

Vice versa, if $z \in Z$, $u \in U' \setminus \langle \e_0 \rangle_{\fqn}$ and $\alpha u \in U'$ for some $\alpha \in \F_{q^t}$, then $\langle z \e_0 + u \rangle_{\fqn} = \langle \alpha z \e_0 + \alpha u \rangle_{\fqn}$.
This proves that
\[
 \ww_{L_U}(\langle z + u \rangle_{\fqn})
 = \dim_{\fq} \{ \alpha \in \F_{q^t} : \alpha u \in U' \}
 = \ww_{L_{U'}}(\langle u \rangle_{\fqn}).
\]
Hence, varying $z$, we see that every point of $L_{U'} \setminus \{ E_0 \}$ gives rise to $|Z|$ points of $L_U \setminus \{ E_0 \}$ of the same weight, and this accounts for all points of $L_U \setminus \{ E_0 \}$.
Points (2) and (4) follow directly from this observation and the fact that $E_0 \in L_U$.
\end{proof}

\begin{remark}
We remark that $L_{U'}$ is contained in $L_U$ and the weight distribution and rank of $L_U$ in the above construction only depends on the weight distribution of $L_{U'}$ and $\ww_{L_{U'}}(E_0)$, but not on the specific structure of $U'$.
In particular, if $\varphi \in \GammaL(d+1,q^t)$, and $\varphi$ fixes $\langle \e_0 \rangle$, then $C_{q,s,t}(Z,U')$ and $C_{q,s,t}(Z,\varphi(U'))$ have the same rank and weight distribution.
\end{remark}

Given some minor conditions, the above construction preserves the property of being $(r,d)$-minimum size.

\begin{proposition}
 \label{prop:liftminconse}
Let $U = C_{q,s,t}(Z,U')$ be as in \Cref{th:contructionfromqt}.
If $L_{U'}$ is an $\fq$-linear set of $(r,d,\Omega)$-minimum size, and $E_0 \in L_{U'} \setminus \Omega$, then $L_U$ is also of $(r,d,\Omega)$-minimum size.
\end{proposition}

\begin{proof}
Suppose that $\Omega = \PG(W,\fqn)$, and that the rank of $L_{U'}$ is $k'$.
Since $L_{U'}$ is of $(r,d,\Omega)$-minimum size, $L_{U'}$ meets $\Omega$ in a canonical subgeometry of $\Omega$, and
\[
\lvert L_{U'} \rvert=q^{k'-1}+\ldots+q^{k'-r}+\lvert L_{\overline{U'}} \rvert,
\]
where $\overline{U'}:=U'+W \leq_q \F_{q^n}^{d+1}/W$.
Since $E_0 \notin \Omega$, up to $\mathrm{GL}(d+1,q^n)$-equivalence, we can suppose that $\Omega$ is defined by the equations $X_0=\ldots=X_{d-r}=0$.
Hence $\F_{q^n}^{d+1}/W$ can be identified with $\F_{q^n}^{d-r+1}$ in an obvious way.
Now, an element $z + u \in U$ belongs to $W$ if and only if $z+ u_0 = u_1 = \ldots = u_{d-r}=0$ if and only if $z = u_0 = \ldots = u_{d-r} = 0$.
Therefore, $U \cap W= U' \cap W$ and so $\Omega$ also meets $L_U$ in a canonical subgeometry.
Moreover, by \Cref{th:contructionfromqt},
\[
\overline{U}=U+W=\{ z + \overline u \colon \overline u \in \overline{U'}\} \subseteq \F_{q^n}^{d+1}/W,
\]
has size $q^{r t}(\lvert L_{\overline{U'}} \rvert -1)+1$.
Therefore we have 
\begin{align*}
 |L_U| & = q^{r t}(\lvert L_{U'} \rvert -1)+1
 = q^{r t}( q^{k'-1}+ \ldots + q^{k'-r}+ \lvert L_{\overline{U'}} \rvert - 1 ) + 1 \\
 & = q^{k-1}+\ldots+q^{k-r}+ \lvert L_{\overline{U}} \rvert,
\end{align*}
with $k = rt + k'$ the rank of $L_U$.
\end{proof}

We can apply \Cref{th:contructionfromqt} with $U'$ as in \Cref{cs:constructionVdV}, obtaining the following families of $d$-minimum size linear sets.

\begin{theorem} 
 \label{cor:newminimumsizebyqt}
Consider $U' = JV_{q,t}(\lambda,t';k_0, \dots, k_d)$ where $t' \mid t$ as in \Cref{cs:constructionVdV}, and choose $\varphi \in \GL(d+1,q^t)$ such that $E_0 \in L_{\varphi(U')}$.
Now define $U = C_{q,s,t}(Z,\varphi(U'))$ as in \Cref{th:contructionfromqt}, with $Z$ an $\F_{q^t}$-subspace of rank $r>0$, not containing 1.
Then $L_U$ is a $d$-minimum size $\fq$-linear set of rank $k = rt + k_0 + \ldots + k_d$.
Moreover, the weight spectrum of $L_U$ is
\[
\begin{cases}
\left(1,\ldots,k_1,k_0,r t+\ww_{L_{\varphi(U')}}(E_0) \right) & \mbox{if } \ww_{L_{\varphi(U')}}(E_0) < k_0 \text{ and } k_1 < k_0, \\
\left(1,\ldots,k_1,r t+\ww_{L_{\varphi(U')}}(E_0) \right) & \mbox{otherwise}.
\end{cases}
\]
\end{theorem}

\begin{proof}
The $\F_q$-linear set $L_{\varphi(U')}$ has the same weight spectrum, weight distribution, and size as $L_{U'}$.
So the assertions follow by applying \Cref{th:contructionfromqt} and \Cref{rk:WeightSpectrumJena}.
\end{proof}

\begin{remark}
 \label{rk:distributionqt}
Using \cite[Remark 2.19]{Jena2021onlinearsets} and \Cref{th:contructionfromqt} (3,4), one could in fact also determine the weight distribution of the linear set in the above theorem.
\end{remark}

The above construction gives new examples of proper $d$-minimum size linear sets. 

\begin{corollary}
In the hypothesis of Theorem \ref{cor:newminimumsizebyqt}, suppose that $L_{U'}$ is a $(d,d,\Pi)$-minimum size $\fq$-linear set, with $\Pi=\PG(W,\F_{q^n})$.
Suppose that $E_0 \in L_{\varphi(U')} \setminus \tilde{\Pi}$, with $\tilde{\Pi} = \PG(\varphi(W),\F_{q^n})$.
Then $L_U$ is a proper $d$-minimum size $\F_q$-linear set in $\PG(d,q^n)$. 
\end{corollary}

\begin{proof}
The linear set $L_{U'}$ is of $(d,d,\Pi)$-minimum size, so the hyperplane $\Pi=\PG(W,\F_{q^n})$ of $\PG(d,q^n)$ meets $L_{U'}$ in a canonical subgeometry of $\Pi$ and
\[
\lvert L_{U'} \rvert=q^{m-1}+\ldots+q^{m-d}+1.
\]
It follows that $\tilde{\Pi}$ also meets $L_{\varphi(U')}$ in a canonical subgeometry of $\tilde{\Pi}$, that is $L_{\varphi(U')}$ is of proper $d$-minimum size as well.
The assertion follows by Proposition \ref{prop:liftminconse}. 
\end{proof}

\Cref{cs:constructionVdV} provides constructions of $d$-minimum size linear sets admitting points of complementary weights. Using \Cref{cor:newminimumsizebyqt}, it is possible to construct proper $d$-minimum size linear sets that do not have this property, as we will see in the next example.
This proves that in general a $d$-minimum size linear set need not contain independent points whose weights sum to the rank of the linear set.
So in general, as already observed in \cite{napolitano2022classification} for the projective line, being minimum size does not determine the weight spectrum and distribution of a linear set.

\begin{example}
Consider 
\[
 U'=JV_{q,6}(\lambda,6;2,2,2)
\]
as in \Cref{cs:constructionVdV}.
Then $L_{U'}$ is an $\F_q$-linear set of rank $6$ in $\PG(2,q^6)$ having size $q^5+q^4+1$ and points of weight at most $2$. Moreover, $\ww_{L_{U'}}(E_0)+\ww_{L_{U'}}(E_1)+\ww_{L_{U'}}(E_2)=2+2+2=6$ is equal to the rank of $L_{U'}$. 
Define
\[
 \varphi \in \GL(3,q^6): (x,y,z) \mapsto (x,y-\lambda x,z).
\]
Then the $\F_q$-linear set $L_{U''}$ in $\PG(2,q^6)$, with
\[
 U''=\varphi(U') 
 = \{ (\alpha_0+\alpha_1 \lambda, \beta_0+\beta_1 \lambda-\alpha_1 \lambda^2, \gamma_0+\gamma_1 \lambda) \colon \alpha_i,\beta_i,\gamma_i \in \F_q \} \subseteq \F_{q^6}^3 
\]
has the same rank, weight spectrum and weight distribution as $L_{U'}$.
Note that $\ww_{L_{U''}}(E_0)=1$.
Choose a 1-dimensional $\F_{q^6}$-subspace $Z \neq \F_{q^6}$ of $\F_{q^{12}}$.
By \Cref{cor:newminimumsizebyqt}, the $\F_q$-linear set $L_U$ of $\PG(2,q^{12})$, with 
\[
U= C_{q,2,6}(Z,U'')
\]
has rank $12$ and size $q^{11}+q^{10}+1$. So, it is a $2$-minimum size linear set.
Note that the weight spectrum of $L_U$ is $(1,2,7)$, and so there do not exist three points of complementary weights.
In particular, $L_U$ cannot be obtained from \Cref{cs:constructionVdV}.
\end{example}

In some cases, \Cref{cor:newminimumsizebyqt} gives us linear sets admitting points of complementary weights, but with a different weight distribution than those of \Cref{cs:constructionVdV}, as stated in the following theorem.

\begin{theorem} \label{th:relationweight}
Consider
\[
 U' = JV_{q,t}(\lambda,t;k_0,\dots,k_d)
\]
and let $\varphi_i$ be the linear map swapping coordinates 0 and $i \in \{0,\dots,d\}$ (with  $\varphi_0 = \mathrm{id}$).
Consider
\[
 U = C_{q,s,t}(Z,\varphi_i(U')),
\]
\begin{enumerate}[(1)]
 \item If $k_i = k_0$, then there exists an $\F_q$-linear set obtained from \Cref{cs:constructionVdV} with the same weight distribution as $L_U$.
 \item If $k_i < k_0-1$, then there does not exist an $\F_q$-linear set obtained from \Cref{cs:constructionVdV} with the same weight distribution as $L_U$.
\end{enumerate}
\end{theorem}
\begin{proof}
Write $n = st$.

(1) If $k_i = k_0$, choose a primitive element $\mu$ of $\fqn$.
Consider $U_2 = JV_{q,n}(\mu,n;k_0+rt,k_1,\dots,k_d)$.
Then $L_U$ and $L_{U_2}$ have the same weight distribution by \Cref{rk:WeightSpectrumJena} (see also \cite[Remark 2.19]{Jena2021onlinearsets}) and \Cref{th:contructionfromqt} (3,4).

(2) Now assume that $k_i < k_0-1$, and suppose that there exists some $\fq$-subspace
\[
 U_3 = JV_{q,n}(\mu,t';k_0',\dots,k_d') \subseteq \F_{q^n}^{d+1}
\]
such that $L_U$ and $L_{U_3}$ have the same weight distribution.
Since $L_U$ has a unique point of weight $r t + k_i$ by \Cref{th:contructionfromqt} (3,4), we see that by \Cref{rk:WeightSpectrumJena}, $k_0' = r t + k_i$.
Furthermore, the second largest weight of $L_U$ and $L_{U_3}$ is respectively $k_0$ and $k_1'$, hence $k_0 = k_1'$.
Let $m'$ denote the number of indices $j$ with $k_j' = k_0$, i.e.\ $k_1' = \ldots = k_{m'}' > k_{m'+1}'$.
Then, using \Cref{rk:WeightSpectrumJena} (see also \cite[Remark 2.19]{Jena2021onlinearsets}),
\[
 N_{k_0}(L_{U_3}) = q^{k_0' - k_1' + 1}\frac{q^{m'}-1}{q-1} = q^{rt + k_i - k_0 + 1} \frac{q^{m'}-1}{q-1}.
\]
On the other hand, by \Cref{th:contructionfromqt} (4) and \Cref{rk:WeightSpectrumJena} (see also \cite[Remark 2.19]{Jena2021onlinearsets}),
\[
 N_{k_0}(L_U) = q^{rt} N_{k_0} (L_{U'}) = q^{rt} \frac{q^m-1}{q-1},
\]
with $m$ the number indices $j$ with $k_j = k_0$.
Therefore,
\[
 q^{rt + k_i - k_0 + 1} \frac{q^{m'}-1}{q-1} = q^{rt} \frac{q^m-1}{q-1}.
\]
Since $\frac{q^m-1}{q-1}$ and $\frac{q^{m'}-1}{q-1}$ are coprime with $q$, this implies that $k_i = k_0 - 1$.
\end{proof}

The case $k_i = k_0 - 1$ is a bit more complicated, and we will not discuss it here.

\subsection{Regarding equivalence}

We show that the two different types of $\fq$-subspaces that define the $d$-minimum size linear sets defined in \Cref{cs:constructionVdV} and \Cref{cor:newminimumsizebyqt} are $\Gamma\mathrm{L}(d+1,q^n)$-inequivalent, even if the associated linear sets have the same weight spectrum and distribution (see \Cref{th:relationweight} (1)), when the dimension of $Z$ is the maximum possible.
The trace function $\mathrm{Tr}_{q^n/q}$ of $\fqn$ over $\fq$, defines a non-degenerate symmetric bilinear form as follows:
\[
(a,b) \in \F_{q^n} \times \F_{q^n} \mapsto \mathrm{Tr}_{q^n/q}(ab) \in \F_q.
\]

Hence, for any subset $S$ of $\fqn$ we can define the orthogonal complement as
\[ 
 S^\perp=\{a \in \fqn \colon \mathrm{Tr}_{q^n/q}(ab)=0,\,\,\,\forall b \in S \}. 
\]
Note that if $S$ is an $\F_{q^t}$-subspace of $\F_{q^n}$, then $S^{\perp}$ is an $\F_{q^t}$-subspace as well.

Given an ordered $\F_{q}$-basis $\mathcal{B}=(\xi_0,\ldots,\xi_{n-1})$ of $\F_{q^n}$, there exists a unique ordered $\fq$-basis $\mathcal{B}^*=(\xi_0^*,\ldots,\xi_{n-1}^*)$ of $\F_{q^n}$ such that $\mathrm{Tr}_{q^n/q}(\xi_i \xi_j^*)= \delta_{ij}$, for $i,j\in\{0,\ldots,n-1\}$, called the \emph{dual basis} of $\mathcal B$, see e.g.\ \cite[Definition 2.30]{lidl1997finite}. 

\begin{lemma}[{\cite[Corollary 2.7]{napolitano2022linearsets}}]
\label{cor:dualbasis}
Let $\lambda \in \fqn$ and suppose that $\mathcal{B}=(1,\lambda,\ldots,\lambda^{n-1})$ is an ordered $\fq$-basis of $\fqn$.
Let $f(x)=a_0+a_1x+\ldots+a_{n-1}x^{n-1}+x^n$ be the minimal polynomial of $\lambda$ over $\fq$. 
Then the dual basis $\mathcal{B}^*$ of $\mathcal{B}$ is 
\[ \mathcal{B}^*=(\delta^{-1}\gamma_0,\ldots,\delta^{-1}\gamma_{n-1}), \]
where $\delta=f'(\lambda)$ and $\gamma_i=\sum_{j=1}^{n-i} \lambda^{j-1}a_{i+j}$, for every $i \in \{0,\ldots,n-1\}$. 
\end{lemma}

\begin{theorem} \label{th:equivalenceminimumsize}
Suppose that $n =(s+1) t$, with $s,t>1$. Consider $U' = JV_{q,t}(\mu,t;k_0, \dots, k_d)$ as in \Cref{cs:constructionVdV}, with $k_0<t-1$. Let $\varphi_i$ be the linear map swapping coordinates 0 and $i \in \{0,\dots,d\}$ (with  $\varphi_0 = \mathrm{id}$) and
define
\[
 U_1 = C_{q,s,t}(Z,\varphi_i(U')),
\]
 as in \Cref{th:contructionfromqt},
with $Z$ an $\F_{q^t}$-subspace of dimension $s$, not containing 1. 
Consider $U_2 = JV_{q,n}(\lambda,n;h_0, k_1, \dots, k_d)$ as in \Cref{cs:constructionVdV}, 
with $h_0=st+k_i$. 
Then the $\fq$-subspaces $U_1$ are $U_2$ are $\Gamma\mathrm{L}(d+1,q^n)$-inequivalent.
\end{theorem}

\begin{proof}
Suppose that $k_i<k_0-1$. Then, by \Cref{th:relationweight},
$L_{U_1}$ and $L_{U_2}$ have a distinct weight distribution, hence $U_1$ and $U_2$ cannot be $\Gamma\mathrm{L}(d+1,q^n)$-equivalent. So suppose that $k_i \in \{k_0-1,k_0\}$ and suppose by contradiction that $U_1$ and $U_2$ are $\Gamma\mathrm{L}(d+1,q^n)$-equivalent via an element $\varphi$.
Since $h_0>k_i$, for every $i \in \{1,\ldots,d\}$, 
the point $E_0$ is the only point in $L_{U_1}$ and in $L_{U_2}$ of weight $h_0$. So, we have that $\varphi(U_1 \cap E_0)=U_2\cap E_0$, that is
\[ aS_1^{\rho}=S_2, \]
some $a \in \fqn^*$ and $\rho \in \mathrm{Aut}(\fqn)$, with $S_1=Z \oplus \langle 1,\mu,\ldots,\mu^{k_i-1} \rangle_{\F_q}$ and $S_2=\langle 1,\lambda,\ldots,\lambda^{h_0-1} \rangle_{\F_q}$.
In particular, we have that $aZ^\rho \subseteq S_2$ and so $(aZ^\rho)^\perp \supseteq S_2^\perp$. Note that $\dim_{\F_{q^{t}}}(aZ^{\rho})=\dim_{\F_{q^{t}}}(Z)=s$. 
This implies that $\dim_{\fq}((aZ^\rho)^\perp)=n-st=t$ and 
hence $(aZ^\rho)^\perp$ is an $\F_{q^{t}}$-subspace of $\fqn$ of dimension one.
Consider the ordered $\F_q$ basis $\mathcal{B}=(1,\lambda,\ldots,\lambda^{n-1})$ of $\F_{q^n}$ and its dual basis $\mathcal{B}^*=(\lambda_0^*,\ldots,\lambda_{n-1}^*)$.
So we have that $S_2^\perp=\langle \lambda_{h_0}^*,\ldots,\lambda_{n-1}^* \rangle_{\fq}$ and since $k_0<t-1$, we have that $h_0<n-1$. By Lemma \ref{cor:dualbasis} it follows that
\[ \lambda_{n-2}^*=\delta^{-1}(a_{n-1}+\lambda), \]
and
\[ \lambda_{n-1}^*=\delta^{-1}, \]
where $f(x)=a_0+a_1x+\ldots+a_{n-1}x^{n-1}+x^n$ is the minimal polynomial of $\lambda$ over $\fq$ and $\delta=f'(\lambda)$.
Now, since $\lambda_{n-2}^*,\lambda_{n-1}^* \in (aZ^\rho)^\perp$ and since $(aZ^\rho)^\perp$ has dimension one over $\F_{q^{t}}$, it follows
\[ \frac{\lambda_{n-2}^*}{\lambda_{n-1}^*}=a_{n-1}+\lambda \in \F_{q^{t}}, \]
that is $\lambda \in \F_{q^{t}}$, a contradiction.
\end{proof}

\section{Below the De Beule-Van de Voorde bound}
 \label{sec:BelowBound}

In this section, we will provide constructions of linear sets $L_U$ in $ \PG(d,q^n)$, with $d>2$, that are of $(r,d)$-minimum size but not of $d$-minimum size.
They have maximum geometric field of linearity $\F_q$, and admit two subspaces of complementary weights.

For our aims, we will suppose that one of these subspaces intersects $L_U$ in a linear set with greater field of linearity.
This gives us the following constructions.

\begin{theorem}
 \label{th:contructionfromqt2}
Let $n=st$, with $s,t>1$, and suppose that
\begin{itemize}
 \item $U_1$ is a $k_1$-dimensional $\F_{q^t}$-subspace of $\fqn^{d_1+1}$,
 \item $U_2$ is a $k_2$-dimensional $\fq$-subspace of $\F_{q^t}^{d_2+1} \subseteq \F_{q^n}^{d_2+1}$.
\end{itemize}
Define $U = U_1 \times U_2$, and $d = d_1 + d_2 + 1$.
Then $L_U$ is an $\fq$-linear set of $\PG(d,q^n)$ of rank $k_1 t + k_2$, with
\[
 |L_U| = |L_{U_1}| + q^{k_1 t} |L_{U_2}|.
\]
Moreover, its weight distribution satisfies
\[
 N_i(L_U) = N_i(L_{U_1}) + q^{k_1 t} N_i(L_{U_2}).
\]
\end{theorem}

\begin{proof}
Take a vector $u \in U_1$ and $v \in U_2$ with $(u,v) \neq \mathbf 0$.
Then
\[
 \ww_{L_U}(\langle (u,v) \rangle_{\fqn})
 = \dim_{\fq} \{ \alpha \in \fqn : \alpha (u,v) \in U \}.
\]
Evidently, $\alpha (u,v) \in U$ if and only if $\alpha u \in U_1$ and $\alpha v \in U_2$.
If $v \neq \mathbf 0$, then $\alpha v \in U_2$ implies that $\alpha \in \F_{q^t}$, and since $U_1$ is an $\F_{q^t}$-subspace, $\alpha u$ is automatically in $U_1$.
Therefore, every point $\langle v \rangle_{\fqn}$ of $L_{U_2}$ gives rise to the $q^{k_1 t}$ points $\{ \langle (u,v) \rangle_{\fqn} : u \in U_1 \}$ of $L_{U}$ with the same weight.
If $v = \mathbf 0$, then we just need that $\alpha u \in U_1$, hence in this way, every point of $L_{U_1}$ gives rise to one point of $L_U$ of the same weight.
Since this accounts for all points of $L_U$, the statement of the theorem follows.
\end{proof}


Using the above theorem, we are able to obtain constructions of linear sets in $\PG(d,q^n)$, with $d \geq 3$, having maximum geometric field of linearity $\F_q$ that are $(r,d)$-minimum size with $2 \leq r<d$ and that are not $d$-minimum size.

\begin{theorem} 
 \label{th:constructionnotdminimum}
Let $n=st$, with $s,t>1$, and suppose that
\begin{itemize}
 \item $U_1$ is a $k_1$-dimensional $\F_{q^t}$-subspace of $\fqn^{d_1+1}$, with $k_1 \leq d_1s$,
 \item $U_2$ is a $k_2$-dimensional $\fq$-subspace of $\F_{q^t}^{d_2+1}$, such that $L_{U_2}$ is a proper $d_2$-minimum size $\fq$-linear set.
\end{itemize}
Define $U = U_1 \times U_2$, $d = d_1 + d_2 + 1$, and $k = k_1 t + k_2$.
Then $L_U$ is a $(d_2,d)$-minimum size $\fq$-linear set of size
\[
 |L_U| = q^{k-1} + q^{k-2} + \ldots + q^{k-d_2} + q^{k_1 t} + |L_{U_1}|.
\]
Hence, $L_U$ is not $d$-minimum size if $k_2 \geq d_2+2$.
Furthermore, if $d_2 \geq 2$, then $\fq$ is the maximum geometric field of linearity of $L_U$.
\end{theorem}

\begin{proof}
The $\F_q$-linear set $L_{U_2} \subseteq \PG(d_2,q^n)$ is of proper $d_2$-minimum size, and so its size is 
\[
 |L_{U_2}| = q^{k_2-1}+\ldots+q^{k_2-d_2}+1.
\]
By Theorem \ref{th:contructionfromqt2}, $L_U$ has rank $k=k_1 t + k_2$ and size 
\[
 \lvert L_{U_1} \rvert + q^{k_1 t}(q^{k_2-1}+\ldots+q^{k_2-d_2}+1)
 = \lvert L_{U_1} \rvert + q^{k-1}+q^{k-2}+\ldots+q^{k-d_2}+q^{k_1 t}.
\]
Moreover there exists a $(d_2-1)$-space $\Gamma=\PG(W,\F_{q^n})$ of $\PG(d_2,q^n)$, with $W \subseteq \F_{q^n}^{d_2+1}$, meeting $L_{U_2}$ in a canonical subgeometry.
Now, let $W'=\{0\}^{d_1+1} \times W$.
Then $W'$ defines a $d_2$-space of $\PG(d,q^n)$ meeting $L_U$ in a canonical subgeometry.
Identifying $\F_{q^n}^{d+1}/W$ with $\F_{q^n}^{d_1+1}$ we have that $\overline{U}=U_1+W \leq_q \F_{q^n}^{d+1}/W$ with
\[
\overline{U}=U_1 \times U',
\]
where $U'$ is an $\F_q$-subspace of $\F_{q^n}$ of dimension $k_2-d_2$.
So again, by Theorem \ref{th:contructionfromqt2}, we have $\lvert L_{\overline{U}}\rvert=q^{k_1 t}+\lvert L_{U_1}\rvert$.
Moreover, by \eqref{eq:card}, $|L_{U_1}| \leq q^{(k_1-1)t} + \ldots +q^t+1$, and since $k_2>d_2+1$ it follows that
\[
 |L_{U_1}| < q^{k-d_2-1} + \ldots + q^{k-d} + 1 - q^{k_1 t}.
\]
This implies that $L_U$ is not of $d$-minimum size.
Finally, the assertion on the geometric field of linearity follows from Remark \ref{rk:geometricfieldminimum}.
\end{proof}

By the above corollary and Proposition \ref{prop:properminimumjena}, we get the following construction.

\begin{corollary}
 \label{cor:underboundwithjena}
Let $n=st$, with $s,t>1$, and suppose that
\begin{itemize}
 \item $U_1 = JV_{q^t,n}(\lambda,n;l_0,\dots,l_{d_1})$, and denote $k_1 = l_0 + \ldots + l_{d_1}$,
 \item $U_2 = JV_{q,t}(\mu,t;m_0,\dots,m_{d_2})$, and denote $k_2 = m_0 + \ldots + m_{d_2}$,
\end{itemize}
with $L_{U_2}$ satisfying the condition of \Cref{prop:properminimumjena}.
Define $U = U_1 \times U_2$.
Then $L_U$ is a $(d_2,d)$-minimum size $\fq$-linear set, but not of $d$-minimum size.
Moreover, if $d_2 \geq 2$, then $\fq$ is the maximum geometric field of linearity of $L_U$.
\end{corollary}

\begin{proof}
Note that
\[
 \lvert L_{U_1} \rvert = q^{(k_1-1) t}+\ldots+q^{(k_1-d_1)t}+1
 < q^{k-d_2-1}+\ldots+q^{k-d}+1 - q^{k_1 t},
\]
and then the assertion follows by Theorem \ref{th:constructionnotdminimum}.
\end{proof}

\begin{remark}
Other examples of $(d_2,d)$-minimum size linear set can be obtained by using the minimum size linear sets constructed in Corollary \ref{cor:newminimumsizebyqt} as $L_{U_1}$ or $L_{U_2}$ in Theorem \ref{th:constructionnotdminimum} .
\end{remark}

\begin{remark}
 It is natural to consider $\PG(d,q^n)$, $n$ not prime, and wonder what the maximal value of $d_2$ is such that the above corollary implies the existence of an $\fq$-linear set in $\PG(d,q^n)$ that is of $(d_2,d)$-minimum size, but not of $d$-minimum size, and has maximum geometric field of linearity $\fq$.
 So let $t$ be the largest proper divisor of $n$.
 Note that $t \geq \sqrt n$.
 We want to construct a set $U_2 = JV_{q,t}(\mu,t;m_0,\dots,m_{d_2})$ with $d_2$ maximal, such that it satisfies the conditions of \Cref{prop:properminimumjena}.
 Hence, there must exist pairwise coprime polynomials $g_i$ of degree $m_i-1$ such that $(m_0-1) + \ldots + (m_{d_2}-1) \leq t-1$.
 Let $\delta(x)$ denote the maximum number of distinct monic irreducible polynomials over $\fq$ such that the sum of their degrees is smaller than $x$.
 Then for any $m \geq 1$, $\delta(q^m) \geq \frac{q^m-1}m$.
 Indeed, consider the minimal polynomials of the elements of $\F_{q^m}^*$.
 Since every element of $\F_{q^m}^*$ is the root of a unique such polynomial, their degrees sum to $q^m-1$.
 Furthermore, the maximum degree equals $m$, so there are at least $\frac{q^m-1}m$ such polynomials.
 Hence, to answer the original question, asymptotically, $d_2 = \Omega(t/\log_q(t)) = \Omega(\sqrt n / \log_q(n))$.
\end{remark}

We conclude this subsection with examples of linear sets of $(1,2)$-minimum size that are not of $(2,2)$-minimum size and have maximum geometric field of linearity $\F_q$.
\begin{proposition} \label{prop:belowdebeuleplane}
   Let $n=st$, with $s > 1$, and $t>2$ prime. Suppose that the smallest prime that divides $s$ is at least $t$. Let
   \begin{itemize}
 \item $U_1$ be a $k_1$-dimensional $\F_{q^t}$-subspace of $\F_{q^n}$,
 \item $U_2 = JV_{q,t}(\mu,t;m_0,m_1)$, with $t = m_0 + m_1$.
\end{itemize}
Define $U = U_1 \times U_2$.
Then $L_U$ is a $(1,2)$-minimum size $\fq$-linear set, but not of $2$-minimum size. Moreover, $\fq$ is the maximum geometric field of linearity of $L_U$.
\end{proposition}
\begin{proof}
By Theorem \ref{th:contructionfromqt2}, $L_U$ is an $\F_q$-linear set of $\PG(2,q^n)$ of rank $(k_1+1)t$ having size 
\[
\lvert L_U \rvert= q^{(k_1+1)t-1}+q^{k_1t}+1,
\]
that is not of $2$-minimum size. Since $L_{U_2}$ has a point of weight $1$, there exists $\varphi \in \mathrm{GL}(2,q^t)$, such that the $\F_q$-linear set $L_{U'}$, with $U'=U_1 \times \varphi(U_2)$ has $E_2$ as a point of weight $1$.
Hence $\F_{q^n}^{3}/E_2$ can be identified with $\F_{q^n}^{2}$ in an obvious way. Clearly, $L_U$ and $L_{U'}$ are $\mathrm{GL}(3,q^n)$-equivalent. In this way, $U'/E_2$ can be identified as an $\F_q$-subspace $\overline{U}=U_1 \times U_2'$, where $U_2'$ is an $(t-1)$-dimensional $\F_q$-subspace of $\F_{q^t}$. Again, by Theorem \ref{th:contructionfromqt2}, we have that $\lvert L_{\overline{U}} \rvert=q^{k_1t}+1$ and hence $L_U$ is a $(1,2)$-minimum size $\F_q$-linear set. 
Suppose now, that $L_U=L_{W}$ for some $\F_{q^r}$-linear set $L_W$. 
If $r < t$, then by our hypothesis, $r$ is coprime with $s$ and $t$, hence $r$ is coprime with $n = st$, and $\F_{q^r}$ is not a subfield of $\fqn$.
Therefore, $r \geq t$.
Let $\ell$ be the line of $\PG(2,q^n)$ having equation $X_0=0$. Then 
    \[
    q^{t-1}+1=\lvert L_{U_2} \rvert = \lvert \ell \cap L_{U}\rvert = \lvert \ell \cap L_{W}\rvert.
    \]
    Since $\ell \cap L_{W}$ is an $\F_{q^r}$-linear set we have that $\lvert \ell \cap L_{W} \rvert \geq q^r+1$. So $t-1 \geq r$, a contradiction.
\end{proof}

\begin{remark}
    Choosing $s,t$ be prime numbers with $s>t>2$, Proposition \ref{prop:belowdebeuleplane} gives examples of $\F_q$-linear sets in $\PG(2,q^{st})$ having rank $st$ and with size $q^{st}+q^{st-t}+1$. Moreover, for such a linear set, the maximum geometric field of linearity if $\F_q$ and there cannot exist a line meeting it in a subline, since its size is less than $q^{st}+q^{st-1}+1$.
\end{remark}

\section*{Acknowledgment}
We would like to thank Jan De Beule, Olga Polverino and Ferdinando Zullo for fruitful
discussions.
Paolo Santonastaso is very grateful for the hospitality of the
Department of Mathematics and Data Science, Vrije Universiteit Brussel, Brussels, Belgium, where he was a visiting PhD student for 2 months during the preparation of this paper.
Paolo Santonastaso was supported by the project ``VALERE: VAnviteLli pEr la RicErca" of the University of Campania ``Luigi Vanvitelli'' and by the Italian National Group for Algebraic and Geometric Structures and their Applications (GNSAGA - INdAM).

\bigskip

\filbreak

\noindent Sam Adriaensen \\
\textit{Vrije Universiteit Brussel} \\
Department of Mathematics and Data Science \\
Pleinlaan 2, 1050 Elsene, Belgium \\
\url{sam.adriaensen@vub.be}

\bigskip \filbreak

\noindent
Paolo Santonastaso \\
\textit{Università degli Studi della Campania ``Luigi Vanvitelli'' }\\
Dipartimento di Matematica e Fisica \\
Viale Lincoln, 5, I– 81100 Caserta, Italy \\
\url{paolo.santonastaso@unicampania.it}

\end{document}